\documentclass[smallextended]{amsart}

\usepackage{amssymb,amsfonts}
\usepackage[all,arc]{xy}
\usepackage{enumerate}
\usepackage{mathrsfs}
\usepackage{amsthm}
\usepackage{standalone}
\usepackage{tikz,pgf} 
\usepackage{graphicx}
\usepackage{caption}
\usepackage[margin=1.2in]{geometry}
 \usepackage{amsaddr}
\usepackage{tikz,pgf} 
\usepackage{mathrsfs}
\usepackage{graphicx}

\newtheorem{thm}{Theorem}[section]
\newtheorem{cor}[thm]{Corollary}
\newtheorem{prop}[thm]{Proposition}
\newtheorem{lem}[thm]{Lemma}

\newtheorem*{rem*}{Remark}

\theoremstyle{definition}
\newtheorem{defn}[thm]{Definition}

\newtheorem{rem}[thm]{Remark}

\numberwithin{equation}{section}

\newcommand{\interior}[1]{%
  {\kern0pt#1}^{\mathrm{o}}%
}

\usepackage{setspace,caption}
\usepackage{lipsum}
\begin{document}

\title{Dihedral branched covers of four-manifolds}

\author{Alexandra Kjuchukova}\footnote{University of Wisconsin--Madison, Department of Mathematics, 480 Lincoln Drive, Madison, WI 53706, USA}
\email{kjuchukova@wisc.edu}

 
\begin{abstract}

Given a closed oriented PL four-manifold $X$ and a closed surface
$B$ embedded in $X$ with isolated cone singularities, we give a formula
for the signature of an irregular dihedral cover of $X$ branched along
$B$. For $X$ simply-connected, we deduce a necessary condition on the
intersection form of a simply-connected irregular dihedral branched
cover of $(X, B)$. When the singularities on $B$ are two-bridge slice,
we prove that the necessary condition on the intersection form of the
cover is sharp. For $X$ a simply-connected PL four-manifold with
non-zero second Betti number, we construct infinite families of
simply-connected PL manifolds which are irregular dihedral branched
coverings of $X$. Given two four-manifolds $X$ and $Y$ whose
intersection forms are odd, we obtain a necessary and sufficient
condition for $Y$ to be homeomorphic to an irregular dihedral $p$-fold
cover of $X$, branched over a surface with a two-bridge slice
singularity.
\end{abstract}

\maketitle

\section{Introduction}

The classification of all branched covers over a given base is a subject
dating back to Alexander, who proved that every closed orientable PL
$n$-manifold is a PL branched cover of $S^{n}$
\cite{alexander1920note}. Alexander's branching sets are PL subcomplexes
of the sphere; he concludes little else about them. Since 1920, the
natural question of how complicated the branching set needs to be, and
how many sheets are needed, in order to realize all manifolds in a given
dimension as branched covers of the sphere, has received much interest -- see, for instance,~\cite{berstein1978degree} and references therein.
It is a famous theorem in dimension 3 that three-fold dihedral covers
branched along knots suffice \cite{hilden1974every},
\cite{hirsch140offene}, \cite{montesinos1974representation}. The
question is considerably more subtle in dimension four. Piergallini and
Iori, among others, have studied the minimal degree needed to realize
all closed oriented PL four-manifolds as covers of the sphere. The
branching sets they consider are either immersed PL submanifolds with
transverse self-intersections or embedded and non-singular PL surfaces.
Piergallini proved in \cite{piergallini1995four} that every closed
oriented PL four-manifold is a four-fold cover of $S^{4}$ branched over
a over a transversally immersed PL surface. He and Iori later refined
this result to show in \cite{iori2002} that singularities can be
removed by stabilizing to a five-fold cover. In light of these universal
realization theorems, one might wish for equally general methods for
obtaining explicit descriptions of the branching sets needed to realize
particular PL four-manifolds as a five-fold covers of~$S^{4}$. It would
also be of interest to better understand the trade-off between
simplifying the branching set and increasing the degree of a cover. Most
recently, Piergallini and Zuddas~\cite{piergallini2016branched} showed
that closed oriented \emph{topological} four-manifolds are also
five-fold covers of the sphere, if one allows for ``wild'' branching sets
with potentially very pathological topology near isolated points. Still,
the complexity of the branching sets near the wild points retains an air
of mystery.

We assume a complementary approach, taking the point of view of studying
all possible covers over a given base $X$ \emph{in terms of the
branching set} and its embedding into the base. As seen from the main
theorem of \cite{viro1984signature}, if $Y$ is a cover of
$S^{4}$ branched over a closed oriented non-singular embedded surface,
then the signature of $Y$ must be zero. Thus, for example, the existing
results on five-fold covers of the four-sphere implicitly make use of
nonorientable branching sets. In contrast, the constructions presented
here make use of branching sets that are oriented surfaces, embedded in
the base piecewise linearly except for finitely many cone singularities.
These ideas have led to new examples of branched covers of $S^{4}$ and
applications to the Slice--Ribbon Conjecture~\cite{cahn2017singular}.
We work with irregular dihedral covers ({Definition~\ref{def:dihedral}}),
which constitute the most direct generalization of the three-dimensional
results of Hilden, Hirsch and Montesinos as well as four-dimensional
results of Montesinos \cite{montesinos1978}. Dihedral covers are also
the ``simplest'' three-fold covers which give rise to interesting examples where the branching sets are singularly embedded (see {Remark~\ref{cyclic-covers}}).

Our results are not restricted to branched covers of the sphere but
apply to any closed oriented four-manifold base. Given an irregular
dihedral branched cover $f: Y\to X$ between two simply-connected
oriented four-manifolds $X$ and $Y$, we relate the intersection forms
of $X$ and $Y$ via $f$. Singularities for us play a central role, and
we compute the signature of a branched cover in terms of data about
the branching set and its singularity.

We begin by defining the type of covers and singularities considered.
Throughout, $D_{p}$~denotes the dihedral group of order $2p$ and $p$ is
odd.

\begin{defn}
\label{def:dihedral}
Let $f: Y\to X$ be a branched cover with branching set $B\subset X$. If
the unbranched cover $f_{|f^{-1}(X-B)}$ corresponds under the classification of covering spaces to $\phi ^{-1}(
\mathbb{Z}/2\mathbb{Z})$ for some surjective homomorphism $\phi : \pi
_{1}(X-B, x_{0})\to D_{p}$, we say that $f$ is an
\textit{irregular dihedral branched cover} of~$X$.
\end{defn}

Put differently, $\phi $ is the monodromy representation of the
unbranched cover associated to $f$ and meridians of the branching set
$B$ map to reflections in the dihedral group $D_{p}$ (thought of as a
subgroup of the symmetric group $S_{p}$). In particular, the existence
of a dihedral cover over a pair $(X, B)$ is a condition on the
fundamental group of the complement of $B$ in $X$. When a (connected)
dihedral cover over the pair $(S^{3}, \alpha )$ exists for some knot
$\alpha $, we say simply that $\alpha $ admits a dihedral cover.

It is helpful to give a description of the pre-images of a point on the
branching set $B$ of an irregular dihedral cover $f: Y\to X$. The
covering space $Y$ is a $\mathbb{Z}/2\mathbb{Z}$ quotient of the
$2p$-fold \emph{regular} dihedral cover $Z$ corresponding to the kernel
of the homomorphism~$\phi $ in {Definition~\ref{def:dihedral}}. For every
locally flat point $b\in B$ the pre-image $f^{-1} (D_{b})$ of a small
neighborhood $D_{b}$ of $b$ in $X$ contains $\frac{p-1}{2}$ components
of branching index $2$ and one component of branching index $1$. The
index 1 component is the fixed set of the involution $Z\to Z$.

\begin{defn}
\label{def:sing}
Let $X$ be a four-manifold and let $B$ be a closed surface embedded in
$X$. Let $\alpha \subset S^{3}$ be a non-trivial knot. For a given point
$z\in B$, assume there exist a small open disk $D_{z}$ about $z$ in
$X$ such that there is a homeomorphism of pairs $(D_{z} - z, B-z)
\cong (S^{3}\times (0, 1), \alpha \times (0, 1))$. We say the
embedding of $B$ in $X$ has a \textit{singularity of type $\alpha $} at
$z$.
\end{defn}

In other words, the knot $\alpha $ is the link of the singularity of
$B$ at $z$. For the covers considered we assume in addition that
singularity is normal, meaning that the pre-image of $z$ under the covering map is a single point. The presence of a singularity $\alpha $ on the branching set $B$ results in a defect, or correction term, to the signature of the covering manifold. While this defect depends only on $\alpha $, it is computed with the help of an associated knot to $\alpha $, defined below.
 
\begin{defn}
\label{defn:chark}
Let $\alpha \subset S^{3}$ and $\beta \subset S^{3}$ be two knot types.
We say that $\beta $ is a \textit{mod $p$ characteristic knot} for
$\alpha $ if there exists a Seifert surface $V$ for $\alpha $ with
Seifert form $L$ such that $\beta \subset V^{\circ }\subset S^{3}$
represents a non-zero primitive class in $H_{1}(V; \mathbb{Z})$ and
$(L+L^T)\beta\equiv 0\text{~mod~}p$. 
\end{defn}

In~\cite{CS1984linking} Cappell and Shaneson defined characteristic
knots and proved that for $p$ a positive odd square-free integer and
$\alpha $ a non-trivial knot, $\alpha $ admits an irregular dihedral
$p$-fold cover if an only if there exists a knot $\beta $ which is a mod~$p$
characteristic knot for $\alpha $. Furthermore, they gave an
explicit construction of a cobordism, here denoted $W(\alpha , \beta
)$, between a dihedral $p$-fold branched cover of $\alpha $ and a cyclic
$p$-fold branched cover of $\beta $. We recall this construction as
needed in the proof of {Proposition~\ref{prop:signatureW}}.

Throughout this article we adopt the following notation. Let
$\chi $ denote the Euler characteristic and $\sigma $ the signature of
a manifold, and let $e$ be the self-intersection number of an embedded
closed submanifold. Given a positive odd integer $p$ and a knot
$\alpha $ in $S^{3}$ which admits an irregular dihedral $p$-fold cover,
denote by $V$ a Seifert surface for $\alpha $ with symmetrized Seifert
pairing $L_{V} : = L + L^{T}$. Let
$\beta \subset V^{\circ }$ be a mod~$p$ characteristic knot for
$\alpha $. Finally, denote by $\sigma _{\zeta ^{i}}$ the Tristram--Levine
$\zeta ^{i}$-signature of a knot \cite{tristram1969some}, where
$\zeta $ is a primitive $p$-th root of unity.

Our first theorem is a necessary condition for the existence of a
$p$-fold irregular dihedral cover $f: Y\to X$ between two four-manifolds
$X$ and $Y$, with a specified embedded surface $B\subset X$ as its
branching set.

\begin{thm}[Necessary condition]\label{thm:necessary}
Let $X$ and $Y$ be closed oriented PL four-manifolds and let $p$ be an
odd prime. Let $B\subset X$ be a closed connected surface, PL-embedded
in $X$ except for an isolated singularity $z$ of type $\alpha $. If
there exists an irregular dihedral $p$-fold cover $f: Y\rightarrow X$
branched along $B$ with a normal singularity at $z$, then the knot
$\alpha $ admits an irregular dihedral $p$-fold cover and this cover is
$S^{3}$. Furthermore, given any corresponding (see footnote~\ref{foot1})
$\mathrm{mod}~p$ characteristic knot $\beta $ for $\alpha $, the following
formulas hold:
\begin{equation}
\label{eq:chi}
\chi (Y)=p\chi (X)-\frac{p-1}{2}\chi (B)-\frac{p-1}{2},
\end{equation}
and
\begin{equation}
\label{eq:sigma}
\sigma (Y)=p\sigma (X) -\frac{p-1}{4}e(B) - \Xi _{p}(\alpha ),
\end{equation}
where
\begin{equation}
\label{eq:xi-def}
\Xi _{p}(\alpha )= \frac{p^{2}-1}{6p}L_{V}(\beta , \beta ) + \sigma (W(
\alpha , \beta )) + \sum _{i=1}^{p-1} \sigma _{\zeta ^{i}}(\beta ).
\end{equation}

\end{thm}

\begin{rem}
The author believes that the above theorem as well as the rest of the
results of this paper extend to the case where $X$ and $Y$ are
topological four-manifolds.
\end{rem}

The main substance of this theorem is finding an expression for
$\Xi _{p}(\alpha )$, the contribution to the signature of $Y$ resulting
from the presence of a singularity of type $\alpha $ on the branching
set. Note that it is straightforward to compute $L_{V}(\beta , \beta
)$ and $ \sum _{i=1}^{p-1} \sigma _{\zeta ^{i}}(\beta )$ from diagrams of
$\alpha $ and $\beta $. A less obvious but essential feature of this
theorem is the fact that the term $\sigma (W(\alpha , \beta ))$ can be
expressed in terms of linking numbers of curves in a dihedral branched
cover of $\alpha $ (see {Proposition~\ref{prop:signatureW}}). A
combinatorial procedure for computing these linking numbers from a
diagram of $\alpha $ is described in Appendix~\ref{appB}, using techniques of
Perko~\cite{perko1964thesis}. This procedure was carried
out in~\cite{cahn2016linking} and implemented in Python.

It is clear from the definition of a characteristic knot that
$\beta $ is not uniquely determined by $\alpha $. While each of the
terms $\frac{p^{2}-1}{6p}L_{V}(\beta , \beta )$, $\sigma (W(\alpha ,
\beta ))$, and $\sum _{i=1}^{p-1} \sigma _{\zeta ^{i}}(\beta )$ depends on
$\beta $, we show in {Proposition~\ref{prop:invariant}} that their sum,
$\Xi _{p}(\alpha )$, is an invariant of $\alpha $ and thus independent
of the choice of characteristic knot.\footnote{Precisely, $\Xi _{p}(
\alpha )$ is an invariant of $\alpha $
\textit{together with a representation} of $\pi _{1}(S^{3} - \alpha , x
_{0})$ onto $D_{p}$. In a lot of cases, the latter is uniquely
determined by $\alpha $, up to the appropriate notion of equivalence.
To each equivalence class of dihedral representations of $\pi _{1}(S
^{3} - \alpha , x_{0})$ corresponds an equivalence class of mod~$p$
characteristic knots for $\alpha $, and $\beta $ can be chosen
arbitrarily within this class. See~\cite{CS1984linking}.\label{foot1}} The author
and Cahn develop a combinatorial method for computing $\Xi _{p}(\alpha
)$ from a Fox $p$-colored diagram of $\alpha $ and apply this method to
specific examples of two-bridge singularities
in~\cite{cahn2017singular}. They also show that, for $\alpha $ a slice
knot which arises as a singularity on a $p$-fold dihedral cover between
four-manifolds, $\Xi _{p}(\alpha )$ gives an obstruction to
$\alpha $ being homotopy ribbon. Precisely, if a slice singularity
$\alpha $ is in fact homotopy ribbon, then $|\Xi _{p}(\alpha )| \leq
\frac{p-1}{2}$ (Theorem~4 of~\cite{cahn2017singular}). Since ribbon
knots are homotopy ribbon, this means in particular that $\Xi _{p}(
\alpha )$ can be used to test potential counter-examples to the
Slice Ribbon Conjecture such as those constructed in~\cite{cha2016casson}.

In the case where the manifold $Y$ are simply-connected,
Equation~{(\ref{eq:chi})} is equivalent to determining the rank of its
intersection form, which is why this (easy to obtain) equation is of
interest. Lastly, we note that {Theorem~\ref{thm:necessary}} generalizes
in the obvious way to the situation where the branching set admits
multiple cone singularities, the signature of the cover picking up a
defect term for each singular point. That is, if the embedding of
$B$ in $X$ has singularities $\alpha _{1}, \dots , \alpha _{k}$, then
\begin{equation*}
\chi (Y)=p\chi (X)-\frac{p-1}{2}\chi (B)-k\frac{p-1}{2}
\end{equation*}
and
\begin{equation*}
\sigma (Y)=p\sigma (X) -\frac{p-1}{4}e(B) -\Sigma _{i=1}^{k} \Xi _{p}(
\alpha _{i}).
\end{equation*}

The following theorem is a partial converse to {Theorem~\ref{thm:necessary}}.

\begin{thm}[Sufficient condition]\label{thm:sufficient}
Let $X$ be a simply-connected closed oriented PL four-manifold. Let
$B\subset X$ be a closed connected surface PL-embedded in $X$ and such
that $\pi _{1}(X-B, x_{0}) \cong \mathbb{Z}/2\mathbb{Z}$. Let $p$ be an
odd prime, and let $\alpha $ be any two-bridge slice knot which admits
a $p$-fold dihedral cover. If $\sigma $ and $\chi $ are two integers
which satisfy Equations~{(\ref{eq:chi})} and {(\ref{eq:sigma})}, respectively,
with respect to $X$, $B$ and $\alpha $, then there exists a
simply-connected four-manifold $Y$ such that $\sigma (Y) = \sigma $,
$\chi (Y) = \chi $ and $Y$ is an irregular dihedral $p$-fold cover of
$X$.
\end{thm}
Note that Equations~{(\ref{eq:chi})} and {(\ref{eq:sigma})} make sense when
$B$ and $\alpha $ are not related. The branching set of the covering map
constructed in the proof of this theorem is a surface $B_{1}\cong B$,
embedded in $X$ with an isolated singularity $z$ of type $\alpha $ and
such that $e(B_{1}) = e(B)$. When restricted to the class of two-bridge
slice singularities, {Theorems~\ref{thm:necessary}
and~\ref{thm:sufficient}} give a necessary and sufficient condition for
a pair of integers $(\sigma , \chi )$ to arise as the signature and
Euler characteristic of a simply-connected dihedral cover over a given
base.

Next, we show that over any indefinite four-manifold $X$, an infinite
family of integer pairs $(\sigma , \chi )$ can be realized as the
signatures and Euler characteristics of simply-connected $p$-fold
dihedral covers over $X$.

\begin{thm}
\label{thm:construction-slice}
Let $X$ be a simply-connected closed oriented PL four-manifold whose
second Betti number is positive. Let $\alpha $ be a two-bridge slice
knot which admits an irregular dihedral $p$-fold cover for $p$ an odd
prime. There exists an infinite family of pairwise non-homeomorphic
simply-connected closed oriented four-manifolds $\{Y_{i}\}_{i=1}^{
\infty }$, each of which is an irregular $p$-fold cover of $X$ branched
over an oriented surface embedded in $X$ with an isolated singularity
of type $\alpha $.
\end{thm}

We remark that for any $p$, infinitely many knots $\alpha $ which
satisfy the hypotheses of the theorem exist, as shown in
{Proposition~\ref{slice_knots_have_covers}}. In~\cite{cahn2017singular},
the author and Cahn used the construction in the proof of
{Theorem~\ref{thm:construction-slice}} to give an infinite family of
three-fold dihedral covers $\mathbb{CP}^{2}\to S^{4}$. Each of these covers is branched along a
two-sphere embedded in $S^{4}$ with one  two-bridge slice
singularity, and the singularities used are pairwise distinct. Note that, as indicated previously, the signature
obstructs the existence of a branched cover $\mathbb{CP}^{2}\to S^{4}$
branched along a non-singular oriented surface.

The construction of infinite families of branched covers given in
{Theorem~\ref{thm:construction-slice}} relies on our ability to vary the
branching set of a dihedral cover. It invites the question, under what
conditions can \textit{a particular manifold} $Y$ be realized as a
$p$-fold dihedral cover over a given base data $(X, B, \alpha )$? In
situations where the manifold $Y$ is (nearly) determined by the rank and
signature of its intersection form, we obtain a complete classification.

\begin{thm}
\label{thm:odd_indefinite}
Let $X$ and $Y$ be simply-connected closed oriented PL four-manifolds
whose intersection forms are odd. Fix an odd square-free integer $p$ and
a two-bridge slice knot $\alpha $ which admits a $p$-fold dihedral
branched cover. Let $B\subset X$ be a PL-embedded surface such that
$\pi _{1}(X-B, x_{0}) \cong \mathbb{Z}/2\mathbb{Z}$. Then, the Euler
characteristic and signature of $Y$ satisfy Equations~{(\ref{eq:chi})}
and~{(\ref{eq:sigma})} with respect to $X$, $B$ and $\alpha $ if and only
if $Y$ is homeomorphic to an irregular $p$-fold dihedral cover of
$X$ branched along a surface $B_{1}$ embedded in $X$ with a singularity
$\alpha $ and such that $B_{1}\cong B$ and $e(B_{1}) = e(B)$.
\end{thm}

The rest of this article is organized as follows. In
Section~\ref{sec:necessary} we prove {Theorem~\ref{thm:necessary}}, give
a formula for $\sigma (W(\alpha , \beta ))$ in terms of linking numbers
in a branched cover of $\alpha $, and show that the defect on the
signature arising from a singularity on the branching set is an
invariant of the singularity type. Section~\ref{sec:sufficient} is
dedicated to the proofs of {Theorems~\ref{thm:sufficient},
\ref{thm:construction-slice} and~\ref{thm:odd_indefinite}}. In Appendix~\ref{appA}
we study characteristic knots of two-bridge knots. Appendix~\ref{appB} lays out
a procedure for calculating linking numbers in a branched cover of a
knot $\alpha $ in terms of a diagram of $\alpha $.

\section{Signatures of dihedral covers}%
\label{sec:necessary}
Our strategy in proving {Theorem~\ref{thm:necessary}} is to resolve the
singularity on the branching set and reduce the computation of the
signature to the case of a PL embedded surface. Then, Novikov
additivity~\cite{novikov1970pontrjagin} implies that the difference between the signatures of the smooth and singular covers is
given by the signature of the manifold used to resolve the singularity.
The final step is to compute the signature of this manifold and prove we can express it in terms of invariants of the singularity type.

\begin{proof}{Proof of {Theorem~\ref{thm:necessary}}}
The assertion that $\alpha $ admits an irregular dihedral $p$-fold cover
and this cover is the three-sphere is verified by considering the local
picture around the singular point. The existence of a $p$-fold dihedral
cover $f: Y\to X$ over the pair $(X, B)$ implies straight away that the
knot $\alpha $ itself admits a $p$-fold dihedral cover $M$. Indeed,
 consider $f^{-1}(\partial N(z))=:M$,
where $z\in B\subset X$ is the singular point on the branching set and
$N(z)$ denotes a small neighborhood. Since by assumption there is a
homeomorphism of pairs
\begin{equation*}
(\partial N(z), B\cap \partial N(z))\cong (S^{3}, \alpha ),
\end{equation*}
the restriction of $f$ to $M$ is a $p$-fold dihedral cover of $\alpha$, as claimed. It is
connected since $z$ is normal. In particular, this means that the knot
group of $\alpha $ surjects onto the dihedral group $D_{p}$.
Furthermore, over $N(z)$ lies the cone on $M$. Since by assumption
$Y$ is a manifold, $M$ is homeomorphic to the three-sphere.

We begin by verifying Equation~{(\ref{eq:chi})}, a straight-forward
computation. Let $N(B)$ denote a regular neighborhood of $B$ in $X$.
Then, we can write
\begin{equation*}
X = (X - N(B))\bigcup _{\partial N(B)} N(B).
\end{equation*}
Since $\partial N(B)$ is a closed oriented three-manifold, we have $\chi (\partial N(B))=0$. This gives:
\begin{equation*}
\chi (X) = \chi (X-N(B)) + \chi (N(B)) = \chi (X-B) + \chi (B).
\end{equation*}
We can further break down this equation as:
\begin{equation*}
\chi (X)
= \chi (X-B) + \chi (B - z) + 1.
\end{equation*}

Similarly, denoting $B':=f^{-1}(B)$ and $z' : = f^{-1}(z)$, we have,
\begin{equation*}
\chi (Y) = \chi (Y-B') + \chi (B' - z') +
1.
\end{equation*}

We know that $f|_{Y-B'}: Y-B'\to X-B$ is a $p$-to-one covering map,
$f|_{B'-z'}: B'-z'\to B-z$ is a $\frac{p+1}{2}$-to-one covering map,
and, of course, $f|_{z'}: z'\to z$ is one-to-one. Therefore,
\begin{equation*}
\chi (Y) = p\chi (X-B) + \frac{p+1}{2}(\chi (B) - 1) +1 = p\chi (X)-
\frac{p-1}{2}\chi (B)-\frac{p-1}{2},
\end{equation*}
as claimed.

Now we turn to the computation of $\sigma (Y)$, a
considerably harder task. We devise a geometric procedure for the
resolution of the singularity on the branched cover. The singularity is
resolved in two stages. At the start, the branching set has one
singular point, in a neighborhood of which the branching set can be
described in terms the knot $\alpha $. Our first step will be to replace
this singularity by a circle's worth of ``standard'' (that is, independent
of the knot type $\alpha $) non-manifold points on the branching set.
The second step will be to excise these ``standard'' singularities and
construct a new cover whose branching set is a PL submanifold
of the base. We carry out these two steps in detail below. In the last stage of the proof, we calculate the effect of the resolution of singularities on the signatures of the four-manifolds involved.

\textit{Step 1.} Let $D_{z}\subset X$ be a neighborhood of the singular
point $z$ such that $D_{z}\cap B$ is the cone on
$\alpha $. As we already established, $\alpha $ admits a $p$-fold
dihedral cover. Equivalently, if $V$ is any Seifert surface for
$\alpha $, there exists a mod~$p$ characteristic knot $\beta \subset
V^\circ$ (see {Definition~\ref{defn:chark}}). Let $W(\alpha , \beta )$ be the
manifold constructed in \cite{CS1984linking} as a cobordism between
a $p$-fold dihedral cover of $(S^{3}, \alpha )$ and a $p$-fold cyclic
cover of $(S^{3}, \beta )$. By construction of $W(\alpha , \beta )$, which is recalled in the
proof of {Proposition~\ref{prop:homologyW}}, there is a $p$-fold branched
covering map
\begin{equation*}
h_{1}: W(\alpha , \beta ) \to S^{3}\times [0, 1],
\end{equation*}
which restricts to a $p$-fold dihedral cover of $(S^{3}\times \{0\}, \alpha )$
and to a $p$-fold cyclic cover of $(S^{3}\times \{1\}, \beta )$. Let
\begin{equation*}
h_{2}: Q\to D^{4}
\end{equation*}
be a $p$-fold cyclic cover of the closed four-ball branched over a
pushed-in Seifert surface $V'$ for $\beta $, as constructed in Theorem~5 of \cite{CS1975branched}. Denote by $\Sigma $ the $p$-fold cyclic
cover of $(S^{3}, \beta )$. By construction, $\partial Q\cong \Sigma
$ and, similarly, $W(\alpha , \beta )$ has one boundary component
homeomorphic to $\Sigma $. Moreover, for $i=1, 2$, the map
\begin{equation*}
h_{i}|_{\Sigma }: \Sigma \to S^{3}
\end{equation*}
is the $p$-fold cyclic cover branched along $\beta $. Therefore, we can
construct a branched cover
%
\begin{equation}
\label{define_h}
h_{1}\cup h_{2}: W(\alpha , \beta )\bigcup _{\Sigma }Q\longrightarrow
S^{3}\times [0, 1]\bigcup _{S^{3}\times \{1\}} D^{4}.
\end{equation}
We denote $W(\alpha , \beta )\bigcup _{\Sigma }Q$ by $Z$ for short, and
the map $h_{1}\cup h_{2}$ by $h$. Thus, we can rewrite
Equation~{(\ref{define_h})} as
\begin{equation*}
h:Z\to D^{4}.
\end{equation*}
This map is a $p$-fold branched cover whose restriction to the boundary
of $Z$ a $p$-fold irregular dihedral cover of $(S^{3}, \alpha )$. So,
denoting the branching set of $h$ by $T$, we have,
%
\begin{equation}
\label{define_T}
T\cong \alpha \times [0, \frac{1}{2}]
\bigcup _{\alpha \times \{\frac{1}{2}\}} V\times \{\frac{1}{2}\}
\bigcup _{\beta \times \{\frac{1}{2}\}} \beta \times [\frac{1}{2}, 1]
\bigcup _{\beta \times \{1\}} V'.
\end{equation}
We see from this description that $T$ is a two-dimensional PL subcomplex
of $D^{4}$ which is a manifold away from the curve $\beta \times \{
\frac{1}{2}\}$. Observe that $\beta \times \{\frac{1}{2}\}$ is embedded
in the interior of $V\times \{\frac{1}{2}\}$. Therefore, in a small
neighborhood of the curve $\beta \times \{\frac{1}{2}\}$, the branching
set is homeomorphic to the Cartesian product of $S^{1}$ and the letter
``$\top $''. (For more details on this construction we once again refer
the reader to \cite{CS1984linking}.)

The idea is to use the map $h$ to construct a new cover of the manifold
$X$ which will differ from the original cover $f$ only over a
neighborhood of the singularity $z\in B$. Specifically, let
$D_{z}':= f^{-1}(D_{z})$ and observe that the restrictions of the maps
$f$ and $h$ to the boundaries of $Y- \interior{D}_{z}$ and $Z$,
respectively, are the $p$-fold irregular dihedral branched
cover\footnote{We use the phrase ``\textit{the} dihedral cover of
$\alpha $'' somewhat liberally throughout this paper. As noted
previously, dihedral covers of $\alpha $ are in bijective correspondence
with equivalence classes of characteristic knots~$\beta $. Naturally,
if $\alpha $ admits multiple non-equivalent dihedral covers, we choose
the one determined by $f$ to construct $Z$.} of $(S^{3}, \alpha )$,
which is again $S^{3}$. We thus obtain a new branched covering
map
%
\begin{equation}
\label{define-f1}
f\cup h: (Y-\interior{D_z'})\bigcup _{S^{3}} Z \longrightarrow (X - \interior{D}
_{z}) \bigcup _{S^{3}} D^{4}.
\end{equation}
Denote the covering manifold $(Y-\interior{D_z'})\bigcup _{S^{3}} Z$ above  by $Y_{1}$ and the map
$f\cup h$ by $f_{1}$. Note that, by Novikov
additivity~\cite{novikov1970pontrjagin}, and since $D_{z}'$ is a
four-ball, $\sigma (Y_{1})=\sigma (Y) + \sigma (Z)$.

Now consider the base space of the branched covering map~\ref{define-f1},
\begin{equation*}
X_{1}:= (X - \interior{D}_{z}) \bigcup _{S^{3}} D^{4}.
\end{equation*}
Since $X_{1}\cong X$, we will continue to denote this space by $X$.
Lastly, denote the branching set of $f_{1}$ by $B_{1}$ and remark that
%
\begin{equation}
\label{define-B1}
B_{1}\cong B - \interior{N(z)} \bigcup _{\alpha \times \{0\}} T.
\end{equation}
In other words, we replaced the cone on $\alpha $ by $T$. As prescribed,
$B_{1}$ has a circle's worth of non-manifold points -- they are the points
corresponding to $\beta \times \{\frac{1}{2}\}$ in
Equation~{(\ref{define_T})} -- regardless of the choice of the knot
$\alpha $.

\textit{Step 2.} Denote by $\beta ^{\ast }$ the curve of non-manifold
points on $T$. We have, $\beta ^{\ast }\subset T\subset D^{4}
\subset X$. Let $N(\beta ^{\ast })$ be a small open tubular neighborhood
of $\beta ^{\ast }$ in $X$. We give $N(\beta ^{\ast })$ the following
trivialization. For every $b\in \beta ^{\ast }$, let $\vec{n}_{1}(b)$ be
the positive normal to $\beta $ in $V$ at the point $b$,
$\vec{n}_{2}(b)$ the positive normal to $V$ in $S^{3}\times \{
\frac{1}{2}\}$, and $\vec{n}_{3}(b)$ the positive normal to
$S^{3}$ in the product structure $S^{3}\times I$. Clearly, $\{\vec{n}
_{1}(b), \vec{n}_{2}(b), \vec{n}_{3}(b)\}$ are linearly independent for
all $b\in \beta ^{\ast }$, so they give a framing of $\beta ^{\ast }$ in
$X$.

We use the above framing to identify $\overline{N(\beta ^{\ast })}$ with
$S^{1} \times B^{3}$ and $\partial \overline{N(\beta ^{\ast })}$ with $\beta ^{
\ast }\times S^{2}$. Now, we construct a new closed oriented
four-manifold, denoted $X_{2}$, as follows:
\begin{equation*}
X_{2} = \bigl (X - N(\beta ^{\ast })\bigr )\bigcup _{S^{1}\times S^{2}}
\bigl (X - N(\beta ^{\ast })\bigr ).
\end{equation*}
The identification of the two copies of $\partial (X - N(\beta ^{
\ast }))$ is done by the homeomorphism
\begin{equation*}
\phi : S^{1}\times S^{2}\to S^{1}\times S^{2}
\end{equation*}
given by the formula
\begin{equation*}
\phi (e^{i\theta }, y)=(e^{-i\theta }, y).
\end{equation*}
In particular, $\phi $ reverses orientation on $S^{1}\times S^{2}$, so
the manifold $X_{2}$ can be given an orientation which restricts to the
original orientations on both copies of $X - N(\beta ^{\ast })$.
Therefore, by Novikov additivity we obtain
%
\begin{equation}
\label{sigma-X}
\sigma (X_{2}) = 2\sigma (X - N(\beta ^{\ast })) = 2\sigma
(X).
\end{equation}
Note that, since $\phi $ restricts to the identity on the $S^{2}$
factor, it identifies the boundary of the branching set $T-N(\beta
^{\ast })$ in one copy of $X - N(\beta ^{\ast })$ with the boundary of
branching set in the other copy of $X - N(\beta ^{\ast })$. Thus, the
image of the branching set after this identification has the form
%
\begin{equation}
\label{surface-mess}
\bigl (B_{1} -N(\beta ^{\ast })\bigr )\bigcup _{3 S^{1}}\bigl (B_{1} -N(
\beta ^{\ast })\bigr ) =: B_{2}.
\end{equation}
Here the fact that the union is taken along three circles corresponds
to the fact that the intersection of $\partial \overline{N(\beta ^{\ast })}$ and $B_{1}$ consists of three closed curves, one for each
``boundary point'' of the letter ``$\top $''.

Note that, since $\phi $ reverses the orientation on each boundary
circle, the orientations of the two copies of $(B_{1} -N(\beta ^{
\ast }))$ can be combined to obtain a compatible orientation on~$B_{2}$. Furthermore, the positive normal to the oriented surface
$(V-N(\beta ^{\ast }))\cup _{\phi _{|}}(V-N(\beta ^{\ast }))$ inside the
three-manifold $(S^{3}\times \{\frac{1}{2}\} -N(\beta ^{\ast }))
\cup _{\phi _{|}} (S^{3}\times \{\frac{1}{2}\} -N(\beta ^{\ast }))$
restricts to the normals to $V$ in each corresponding copy of
$S^{3}$. This observation will come into use shortly.

Recalling the definition of $B_{1}$, namely $B_{1}=(B-\interior{D}
_{z})\bigcup _{\alpha }(T-N(\beta ^{\ast }))$, we can describe
$B_{2} = (B_{1}-N(\beta ^{\ast }))\cup _{3S^{1}}(B_{1}-N(\beta ^{\ast }))$
in more detail as follows:
%
\begin{equation}
\label{surface-bigger-mess}
B_{2}=\bigl ( (B-\interior{D}_{z})\cup _{\alpha }(T-N(\beta ^{\ast }))
\bigr ) \bigcup _{3 S^{1}}\bigl ((B-\interior{D}_{z})\cup _{\alpha }(T-N(
\beta ^{\ast }))\bigr ).
\end{equation}
By construction, $B_{2}$ is embedded piecewise-linearly in $X_{2}$ --
that is, all singularities have been resolved. In addition,
$B_{2}$ has two connected components, since deleting a neighborhood of
$\beta ^{\ast }$ disconnects $T$. Thus, two copies of $(T-N(\beta ^{\ast }))$ gives four disjoint surjaces with boundary. Attaching along the
three curves in $(S^{1}\times S^{2})\cap (T-N(\beta ^{\ast }))$ via $\phi _{|}$ has the effect of pairing off each of these four surfaces with
boundary and its homeomorphic copy. This produces two closed surfaces which
we denote $B_{2}'$ and $B_{2}''$. Here, $B_{2}'$ is the component of
$B_{2}$ obtained by identifying two copies of $(B-\interior{D}_{z})
\cup _{\alpha }(V-\beta )$ along $S^{1}\amalg S^{1}$, where $(V-\beta )$
denotes the complement in $V$ of a small open neighborhood of
$\beta $. In turn, $B_{2}''$ is the component of $B_{2}$ obtained by
identifying two copies\footnote{It would be more consistent with our
earlier notation to say that $B_{2}''$ is obtained from two copies of
$\beta \times [\frac{1}{2},1]\cup _{\beta \times \{1\}}V'$, which, of
course, is a surface homeomorphic to $V'$.} of $V'$ along $S^{1}$. By
construction, the cover over $B_{2}'$ is $p$-fold dihedral, whereas the
cover over $B_{2}''$ is $p$-fold cyclic. That is, a point in
$B_{2}'$ has $\frac{p+1}{2}$ pre-images, all but one of branching index
2, whereas a point in $B_{2}''$ has one pre-image of index $p$. This
distinction becomes relevant when we compute the signature of the
branched cover whose branching set consists of $B_{2}'$ and
$B_{2}''$ (Equation~{(\ref{eq1})}).

Now our aim is to construct a $p$-fold branched cover of $(X_{2}, B
_{2})$ from the covers $f$ of $(X, B)$ and $h$ of $(D^{4}, T)$.
Moreover, we require that the cover we construct be a manifold. We are
helped greatly in this task by the observation that
\begin{equation*}
N':= h^{-1}(N(\beta ^{\ast }))\cong S^{1}\times B^{3},
\end{equation*}
even though the branching set of $h$ is non-manifold along $\beta ^{
\ast }$. (A nice explanation of this rather surprising fact can be found
on pp.~173--174 of~\cite{CS1984linking}.) Here, the trivialization of the
normal bundle $N'$ of $h^{-1}(\beta ^{\ast })$ is given by the pull-back
of the trivialization of $N(\beta ^{\ast })$. We construct the
covering manifold as
\begin{equation*}
Y_{2}:=\bigl (Y_{1}-N'\bigr )\bigcup _{S^{1}\times S^{2}}\bigl (Y_{1}-N'
\bigr ).
\end{equation*}
The identification along the boundary $S^{1}\times S^{2}$ is again done
by $\phi $, so, once more, $Y_{2}$~can be given an orientation which
restricts in each copy of $(Y_{1}-N')$ to the orientation compatible
with the given orientation on $Y$. In particular,
%
\begin{equation}
\label{signature-Y2}
\sigma (Y_{2})=\sigma \bigl ((Y_{1}-N')\cup _{S^{1}\times S^{2}}(Y_{1}-N')
\bigr ) = 2\sigma (Y_{1}) = 2(\sigma (Y)+ \sigma (Z)).
\end{equation}
Recall that both $Y_{2}$ and $X_{2}$ were constructed from copies of
$(Y_{1}-N')$ and $(X - N(\beta ^{\ast }))$ by gluing via $\phi$. Therefore, the
restrictions of $f_{1}$ to the two copies of $(Y_{1}-N')$,
\begin{equation*}
f_{1}|:(Y_{1}-N')\to (X - N(\beta ^{\ast })),
\end{equation*}
can be glued to obtain a map
\begin{equation*}
f_{2}: \bigl ( (Y_{1}-N')\cup _{S^{1}\times S^{2}}(Y_{1}-N')\bigr )
\to \bigl (X - N(\beta ^{\ast })\bigr )\cup _{S^{1}\times S^{2}} \bigl (X -
N(\beta ^{\ast })\bigr ),
\end{equation*}
written for short as
\begin{equation*}
f_{2}: Y_{2} \to X_{2}.
\end{equation*}
This is the branched cover we will use in the final step of the proof
to compute the signature of the original manifold $Y$.

\emph{Step 3.} To complete the proof, what remains is to compute the
effect this construction has on the signatures of the base and covering
manifolds. By Viro's formula~\cite{viro1984signature} for the
signature of a branched cover, we have
%
\begin{equation}
\label{eq1}
\sigma (Y_{2}) = p\sigma (X_{2}) - \frac{p-1}{4}e(B_{2}') - \frac{p
^{2}-1}{3p}e(B_{2}'').
\end{equation}
Recall that from Equations~{(\ref{sigma-X})} and~{(\ref{signature-Y2})} we have
%
\begin{equation}
\label{eq2}
\sigma (X_{2})=2\sigma (X)
\end{equation}
and
%
\begin{equation}
\label{eq3}
\sigma (Y)=\frac{1}{2}\sigma (Y_{2})-\sigma (Z).
\end{equation}
Also, by Novikov additivity,
%
\begin{equation}
\label{eq4}
\sigma (Z, S^{3}) = \sigma (W(\alpha ,\beta ), S^{3}\cup \Sigma ) +
\sigma (Q, \Sigma ) = \sigma (W(\alpha ,\beta )) + \sum _{i=1}^{p-1}
\sigma _{\zeta ^{i}}(\beta ).
\end{equation}
In the last step, we have expressed the signature of $Q$ in terms of
Tristram--Levine signatures of $\beta $, using Theorem~5
of~\cite{CS1975branched}. We have also shortened $\sigma (W(\alpha
,\beta ), S^{3}\cup \Sigma )$ to $\sigma (W(\alpha ,\beta ))$. Now we
combine Equations~{(\ref{eq1})}, {(\ref{eq2})}, {(\ref{eq3})} and {(\ref{eq4})} in order
to express $\sigma (Y)$ in terms of data about $X$, the branching set,
the singularity $\alpha $ and its characteristic knot $\beta $. After
simplifying, we obtain,
%
\begin{equation}
\label{euler-mess}
\sigma (Y)=p\sigma (X) -\frac{1}{2}\bigl (\frac{p-1}{4}e(B_{2}') + \frac{p
^{2}-1}{3p}e(B_{2}'')\bigr ) -\sigma (W(\alpha , \beta )) - \sum _{i=1}
^{p-1} \sigma _{\zeta ^{i}}(\beta ).
\end{equation}
To arrive at Equation~{(\ref{eq:sigma})}, what remains is to compute the
self-intersection numbers of $B_{2}'$ and $B_{2}''$ in $X_{2}$ and
relate them to that of $B$ in $X$.

We denote the intersection number of two submanifolds by ``$\centerdot
$'', and the push-off of a submanifold $S$ along a normal $\vec{u}$ by
$\vec{u}(S)$. For brevity, we also denote $B-\interior{D}_{z}$, the
complement in $B$ of a small open neighborhood of the singularity
$z$, by $B_{z}$.

Note that if $\vec{v}$ is a continuous extension (not necessarily
non-vanishing) to $B_{z}$ of the normal to $V$ in $S^{3}=\partial D
_{z}$ such that $B_{z}$ and $\vec{v}(B_{z})$ are transverse, then by
definition
%
\vspace*{-2pt}
\begin{equation}
\label{define-e}
e(B)=(B_{z}\cup _{\alpha }V)\centerdot \vec{v}(B_{z}\cup _{\alpha }V).
\end{equation}
Now, $V\subset D_{z}$ and $\vec{v}(V)\subset D_{z}$, whereas
$B_{z}\cap D_{z} = \alpha $ and $\vec{v}$ can be chosen so that
$\vec{v}(B_{z})\cap D_{z}=\vec{v}(\alpha )$. In particular, $V$ is
disjoint from both $\vec{v}(V)$ and $\vec{v}(B_{z})$, and $B_{z}$ is
disjoint from $\vec{v}(V)$. Therefore, Equation~{(\ref{define-e})}
simplifies to
%
\vspace*{-2pt}
\begin{equation}
\label{simpler}
e(B)=B_{z}\centerdot \vec{v}(B_{z}),
\end{equation}
where the right hand side represents the intersection number of
transverse submanifolds with disjoint boundary in $X-\interior{D}_{z}$.
Recall that the surface $B_{2}'$ is obtained from two copies of
$B_{z}\cup _{\alpha }(V - \beta )$ attached by a homeomorphism
$\phi _{|}$ on their boundary $\beta _{1} \amalg \beta _{2}$, reversing
orientation on each component. Recall also that the restriction to
$S^{3}\times \{\frac{1}{2}\} \cap \partial N(\beta ^{\ast })$ of the
positive normal to $V$ in $S^{3}\times \frac{1}{2}$ (and thus of
$\vec{v}$), is preserved by the gluing homeomorphism~$\phi _{|}$.
Therefore, the two copies of the normal $\vec{v}$ to $B_{z}\cup _{
\alpha }(V - \beta )$ can be combined obtain a normal, which we also
denote $\vec{v}$, to $B_{2}'$ in $X_{2}$. Then,
%
\vspace*{-2pt}
\begin{equation}
\label{B2}
B_{2}'= B_{z}\cup _{\alpha }(V-\beta ) \cup _{\beta _{1}\amalg \beta _{2}}
(V-\beta )\cup _{\alpha }B_{z}.
\end{equation}
Since by the argument above $V-\beta $ and $\vec{v}(V-\beta )$
contribute nothing to the self-intersection $B_{2}'\centerdot \vec{v}(B
_{2}')$, we have
%
\vspace*{-2pt}
\begin{equation}
\label{eB2'}
e(B_{2}')= 2(B_{z}\centerdot \vec{v}(B_{z})) = 2e(B).
\end{equation}

Next, recall that $B_{2}''=V'\cup _{\beta }V'$ and $\vec{n}_{1}$ is the
normal to $\beta $ in $V$. Denote by $\vec{v'}$ a continuous extension
(not necessarily nowhere-zero) to $V'$ of the normal $\vec{n}_{1}$ so
that $V'$ and $\vec{v'}(V')$ are transverse. We have,
%
\vspace*{-2pt}
\begin{equation}
\label{eB2''}
e(B_{2}'')= 2(V'\centerdot \vec{v'}(V')) = 2lk(\beta , \vec{v'}(
\beta ))=2lk(\beta , \vec{n}_{1}(\beta ))= L_{V}(\beta , \beta ).
\end{equation}
Recall that $L_{V}$ denotes the symmetrized linking form on $V$, the
Seifert surface for $\alpha $. The last equality follows from the fact
that $\vec{n}_{1}$ and $\vec{n}_{2}$, the normal to $\beta $ determined
by~$V$, are everywhere linearly independent, so $lk(\beta , \vec{n}
_{1}(\beta ))= lk(\beta , \vec{n}_{2}(\beta ))$.

Substituting for $e(B_{2}')$ from Equation~{(\ref{eB2'})} and for
$e(B_{2}'')$ from Equation~{(\ref{eB2''})}, we can rewrite
Equation~{(\ref{euler-mess})} as
%
\vspace*{-2pt}
\begin{equation}
\label{signY}
\begin{split}
\sigma (Y) &= p\sigma (X) -\frac{p-1}{4}e(B) - \frac{p^{2}-1}{6p}L_{V}(
\beta , \beta ) - \sigma (W(\alpha , \beta )) - \sum _{i=1}^{p-1}
\sigma _{\zeta ^{i}}(\beta ) \\ & = p\sigma (X) -\frac{p-1}{4}e(B) - \Xi_p(\alpha).
\end{split}
\end{equation}
With that, the proof is complete.
\goodbreak
\end{proof}

\begin{rem}
The property that a $p$-fold dihedral cover of a knot $\alpha $ is
homeomorphic to the three-sphere can be regarded as a condition for
$\alpha $ to be an allowable singularity on the branching set of an
irregular $p$-fold dihedral cover between four-manifolds. The condition
is satisfied, for example, for all two-bridge knots and any odd $p$ (see
the proof of {Lemma~\ref{fox-milnor}}) and can be disregarded if one
allows the covering space to be a stratified space. In this case, an
analogous formula can be obtained, using intersection homology signature
or Novikov signature of a singular space.
\end{rem}

\begin{rem}
\label{cyclic-covers}
It is natural to consider computing signatures of cyclic branched covers using the same ideas. Indeed, our techniques would apply and the arguments would be considerably simpler: in the notation
of the proof of the last theorem, only the manifold $Q$ would be needed
to resolve the singularity. However, it is a consequence of the Smith Conjecture~\cite{morgan1984smith} that no non-trivial
knot can arise as a singularity on a cyclic cover between
four-manifolds. This is why the case of cyclic covers is not
considered in this work. Our methods, however, are applicable in a
scenario where stratified spaces are allowed as the covers.
\end{rem}

Although we have arrived at the desired equation, the formula we have
obtained does not quite startle with its usefulness, as long as the term
$\sigma (W(\alpha , \beta ))$ remains obscure. As stated in the
introduction, we aim to compute the defect to $\sigma (Y)$ in terms of
the singularity type $\alpha $. That is, we need to express
$\sigma (W(\alpha , \beta ))$ explicitly in terms of some computable
invariants of $\alpha $. It turns out that we can give a formula for
$\sigma (W(\alpha , \beta ))$ using linking numbers of curves in the
irregular dihedral $p$-fold branched cover of $\alpha $. To this end,
we first compute the second homology group of this manifold: this is
the content of {Proposition~\ref{prop:homologyW}}. In
{Corollary~\ref{basis}}, we give a basis for this homology group in terms
of lifts to a dihedral cover of $\alpha$ of curves in the chosen Seifert surface $V$. Finally, in
{Proposition~\ref{prop:signatureW}}, we give an explicit formula for the
term $\sigma (W(\alpha , \beta ))$ using linking numbers of the above
curves.

\begin{prop}
\label{prop:homologyW}
Let $\alpha \subset S^{3}$ be a knot which admits a $p$-fold irregular
dihedral cover $M$ for some odd prime $p$. Let $V$ be a Seifert surface
for $\alpha $ and let $\beta \subset V$ be a $\mathrm{mod}\ p$ characteristic knot
for $\alpha $. Let $\Sigma$ be the $p$-fold cyclic cover of $\beta $. Let
$W(\alpha , \beta )$, here denoted $W$, be the cobordism between $M$ and
$\Sigma $ constructed in~\cite{CS1984linking} and used in the proof
of {Theorem~\ref{thm:necessary}}. Denote by $V-\beta $ the surface $V$
with a small open neighborhood of $\beta $ removed, and by $\beta _{1}$
and $\beta _{2}$ the two boundary components of $V-\beta $ that are
parallel to $\beta $. Then
%
\begin{equation}
H_{2}(W, M; \mathbb{Z})\cong \mathbb{Z}^{\frac{p-1}{2}} \oplus (H_{1}
(V-\beta ; \mathbb{Z}) /([\beta _{1}], [\beta _{2}]))^{\frac{p-1}{2}}.
\end{equation}
\end{prop}

\begin{proof}
Since Cappell and Shaneson's construction of $W$ is essential to our
computation, we review it here. Let $f:\Sigma \rightarrow S^{3}$ be the
cyclic $p$-fold cover of $\beta $. Since $p$ is prime, it is well known
that $\Sigma $ is a rational homology sphere~\cite{rolfsen1976knots}.
Let
\begin{equation*}
f\times 1_{I}: \Sigma \times [0, 1]\rightarrow S^{3}\times [0,1]
\end{equation*}
be the induced cyclic branched cover of $S^{3}\times [0,1]$. Next, let
\begin{equation*}
J:= f^{-1} (V\times [-\epsilon , \epsilon ]\times \{1\})
\end{equation*}
be the pre-image of a closed tubular neighborhood $V\times [-\epsilon
, \epsilon ]\times \{1\}$ of $V\times \{1\}$ in $S^{3}\times \{1\}$, and
let $T$ be its ``core'', namely $T:= f^{-1} (V\times \{0\}\times \{1\})$
with
\begin{equation*}
T\subset J\subset \Sigma \times \{1\}.
\end{equation*}
Then $J$ deformation-retracts to $T$, and $T$ consists of $p$ copies of
$V$ identified along $\beta $ and permuted cyclically by the group of
covering transformations of $f$.

Consider the involution $\bar{h}$ of $J$ defined in
\cite{CS1984linking} as the lift of the map
\begin{equation*}
h: V\times [-\epsilon , \epsilon ]\rightarrow V\times [-\epsilon ,
\epsilon ],
\end{equation*}
\begin{equation*}
h(u,t)\mapsto (u,-t)
\end{equation*}
fixing a chosen copy of $V$ in $f^{-1}(V\times \{0\}\times \{1\})$. Let
$q$ be the quotient map defined as
\begin{equation*}
q: \Sigma \rightarrow \Sigma /\{x\sim \bar{h}(x) | x\in J\} = \Sigma
/ \bar{h}.
\end{equation*}
Similarly, let
\begin{equation*}
W:= (\Sigma \times I) / (x\sim \bar{h}(x) | x\in J\subset \Sigma
\times \{1\})
\end{equation*}
and let
\begin{equation*}
M:=(\Sigma -J^{\circ })/\bar{h}.
\end{equation*}
By definition,
$\Sigma /\bar{h} = M\cup (J/\bar{h})$. As shown in
\cite{CS1984linking}, $M$ is the $p$-fold irregular dihedral cover of
$\alpha $ and $W$ is a cobordism between $M$ and the cyclic $p$-fold
cover $\Sigma = \Sigma \times \{0\}$ of $\beta $.

This completes the description of the construction of the pair
$(W, M)$ whose second homology group we are about to compute.

Note that $W$ is by definition the mapping cylinder of the quotient map
$q$. Let $R:= J/ \bar{h}$. We have
\begin{equation*}
H_{2}(W, M;\mathbb{Z})\cong H_{2}(M\cup R, M;\mathbb{Z})\cong H_{2}(R,
M\cap R;\mathbb{Z}),
\end{equation*}
where the second isomorphism is excision, and the first follows from the
fact that $W$ deformation-retracts onto $\Sigma / \bar{h}= M\cup R$.
Since
\begin{equation*}
M\cap R = \partial R - V_{0}
\end{equation*}
(following the notation of \cite{CS1984linking}, $V_{0}$ is the copy
of $V$ in $T$ fixed by $\bar{h}$), we can rewrite the above isomorphism
as
%
\begin{equation}
\label{eq:5}
H_{2}(W, M;\mathbb{Z})\cong H_{2}(R, \partial R - V_{0};\mathbb{Z}).
\end{equation}
The relevant portion of the long exact sequence of the pair
$(R, \partial R - V_{0})$ is
%
\begin{equation}
\label{LES-pair}
H_{2}(R;\mathbb{Z}) \rightarrow H_{2}(R, \partial R - V_{0};
\mathbb{Z})\rightarrow H_{1}( \partial R - V_{0};\mathbb{Z}) \rightarrow
H_{1}(R;\mathbb{Z}).
\end{equation}

We will shortly show that $H_{2}(R;\mathbb{Z}) = 0$ (see
Equation~{(\ref{H2R})}). Assuming this for the moment, the above exact
sequence, combined with Equation~{(\ref{eq:5})}, gives
%
\begin{equation}
H_{2}(W, M;\mathbb{Z})
\cong \ker (i_{\ast }: H_{1}( \partial R - V
_{0};\mathbb{Z}) \rightarrow H_{1}(R;\mathbb{Z}) ).
\end{equation}

Our goal, therefore, is compute this kernel. Note, furthermore, that we
are not simply interested in its rank over $\mathbb{Z}$; we want to
write down an explicit basis for $\ker (i_{\ast })$ in terms of lifts
to $M$ of curves in the complement of $\alpha \subset S^{3}$.

Recall that $V$ is a surface with boundary and that, by definition,
$\beta $ represents a non-zero primitive class in $H_{1}(V;
\mathbb{Z})$. Therefore, $\beta $ can be completed to a one-dimensional
subcomplex $C\vee \beta $ which $V$ deformation-retracts to, where
$C$ is the wedge of $2g-1$ circles, and $g$ the genus of $V$. Moreover,
we can perform the deformation retraction of $V$ onto such a one-complex
simultaneously on each copy of $V$ contained in $T$, fixing the curve
of intersection $\beta $. Therefore, $T$ deformation-retracts to a
one-complex containing $\beta $ wedged to $p$ copies of $C$, where
\begin{equation*}
H_{1}(C;\mathbb{Z})\cong H_{1}(V;\mathbb{Z})/[\beta ].
\end{equation*}
It follows that
\begin{equation*}
H_{2}(J; \mathbb{Z})\cong H_{2}(T; \mathbb{Z})\cong 0
\end{equation*}
and
\begin{equation*}
H_{1}(J; \mathbb{Z})\cong H_{1}(T; \mathbb{Z})\cong \oplus _{p} (H_{1}(V;
\mathbb{Z}) / [\beta ])\oplus \mathbb{Z},
\end{equation*}
where the singled-out copy of $\mathbb{Z}$ is generated by $[\beta ]$.

Furthermore, since the deformation--retraction of $J$ onto $T$ can be
chosen to commute with $\bar{h}$, $J/\bar{h}=R$ deformation-retracts to
$T/\bar{h}$, which is isomorphic to $\frac{p+1}{2}$ copies of $V$
identified along $\beta $. (This isomorphism is seen from the fact that
$V_{0}$ is fixed by $\bar{h}$, and the remaining $p-1$ copies of $V$ in
$T$ become pairwise identified in the quotient. All copies of
$\beta $ are identified to a single circle in both $T$ and $T/\bar{h}$.)
Therefore,
%
\begin{equation}
H_{1}(R;\mathbb{Z})\cong (H_{1}(V;\mathbb{Z})/[\beta ])^{
\frac{p+1}{2}}\oplus \mathbb{Z}.
\end{equation}
By the same reasoning as above, we can also conclude that $T/\bar{h}$
deformation-retracts to a one-complex, so, as claimed,
%
\begin{equation}
\label{H2R}
H_{2}(R; \mathbb{Z})=H_{2}(J/\bar{h}; \mathbb{Z})\cong H_{2}(T/
\bar{h}; \mathbb{Z})\cong 0.
\end{equation}

Next, we examine $\partial (J)$ and $\partial R$. To start,
$\partial (V\times [0,1]) \cong V\cup _{\alpha }V$. Therefore,
$\partial (J)$ can be thought of as the union of $p$ copies of
$(V-\beta )\cup _{\alpha }(V-\beta )$, which we label $V_{i}^{+}
\cup _{\alpha _{i}} V_{i}^{-}$, $0\leq i<p$, with further identifications
we now describe. Denote the copy of $\beta $ lying in $V_{i}^{\pm }$ by
$\beta _{i}^{\pm }$, and cut each $V_{i}^{\pm }$ along $\beta _{i}^{
\pm }\subset V_{i}^{\pm }$. Now, $V_{i}^{\pm }- \eta (\beta _{i}^{
\pm })$ is a connected surface with three boundary components,
$\alpha _{i}$, $\beta _{i, 1}^{\pm }$ and $\beta _{i, 2}^{\pm }$, where the
$\beta _{i,j}^{\pm }\subset V_{i}^{\pm }$ are labeled in such a way that
$\beta _{i,j}^{+}$ and $\beta _{i,j}^{-}$ correspond to $\beta _{i, j}
\times \{1\}$ and $\beta _{i, j}\times \{0\}$ in $V_{i}\times [0, 1]$;
that is, the projection map $V_{i}\times [0, 1]\to V_{i}$ sends
$\beta _{i,j}^{+}$ and $\beta _{i,j}^{-}$ to the same boundary component
of $V-\eta (\beta )$.

Now, the covering translation $\tau $ on $J$ acts in a neighborhood of
$\beta $ as rotation by $\frac{2\pi i}{p}$ degrees. With the labeling
described above, $\tau $ carries $\beta _{i, j}^{\pm }$ to $
\beta _{i+1\text{mod~}p, j}^{\pm }$. Moreover, in $\partial (J)$, we have the
following identifications: $\beta _{i,1}^{+}\sim
\beta _{i+\frac{p+1}{2}\text{mod~}p,2}^{+}$ and $\beta _{i,1}^{-}\sim
\beta _{i+\frac{p-1}{2}\text{mod~} p,2}^{-}$. Put differently, we can think of
$\partial (J)$ as obtained from $2p$ disjoint copies of $V-\beta $,
labeled $V_{i}^{\pm }- \beta _{i}^{\pm }$, by gluing $\alpha _{i}^{+}$ to
$\alpha _{i}^{-}$ and the $\beta _{i, j}^{\pm }$-s according to the
identifications specified above. Thus, $\partial (J)$ is a closed
surface. By considering its Euler characteristic, we find that its genus
is $(2g-1)p + 1$. In addition, from the above decomposition of
$\partial (J)$ we find that
%
\begin{equation}
H_{1}(\partial (J); \mathbb{Z})\cong ((H_{1}(V - \beta ; \mathbb{Z}))
/ ( [\beta _{1}], [\beta _{2}] ))^{2p}\oplus \mathbb{Z}^{2p+2}.
\end{equation}

Recall that $R$ is a $\mathbb{Z}/2\mathbb{Z}$ quotient of $J$, where the
$\mathbb{Z}/2\mathbb{Z}$ action fixes $V_{0}\times I$ and pairs off
$V_{i}^{+}$ with $V_{p-i}^{-}$ for $1\leq i\leq \frac{p-1}{2}$. It
follows that $\partial R - V_{0}$ is a surface of genus $p(g-1)+
\frac{p+1}{2}$ and we have,
%
\begin{equation}
\label{homology-boundary}
H_{1}(\partial R - V_{0}; \mathbb{Z})\cong ((H_{1}(V - \beta ;
\mathbb{Z})) / ([\beta _{1}], [\beta _{2}]))^{p}\oplus \mathbb{Z}^{p+1}.
\end{equation}

Recall that our aim is to compute
\begin{equation*}
\ker (i_{\ast }: H_{1}( \partial R - V_{0};\mathbb{Z}) \rightarrow H
_{1}(R;\mathbb{Z})).
\end{equation*}

Again, the idea behind writing $H_{1}(\partial R - V_{0};
\mathbb{Z})$ as in Equation~{(\ref{homology-boundary})} is to obtain a
convenient basis for this kernel, and to relate this basis to a basis
for the homology of $V$. Specifically, a virtue of our expression for $H_{1}( \partial R - V_{0};\mathbb{Z})$ is that a basis
for $H_{1}(V -
\beta ; \mathbb{Z})) / ([\beta _{1}], [\beta _{2}])$ can be extended to a basis for both $H_{1}(V;
\mathbb{Z})$ and $H_{1}(V-\beta ; \mathbb{Z})$. In particular, the
inclusion $i:\partial R-V_{0}\to R$ induces an injection ${i_{\ast }}
_{|}: H_{1}(V - \beta ; \mathbb{Z}) / ([\beta _{1}], [\beta _{2}])
\to H_{1}(R; \mathbb{Z})$ for each copy of $V - \beta \subset \partial
R-V_{0}$. Furthermore, for each $V_{j}\subset T/\bar{h}\simeq R$, the
inclusion $i: V_{j} \to R$ induces an injection on $H_{1}(V - \beta ;
\mathbb{Z}) / ([\beta _{1}], [\beta _{2}])$.

With this in mind, using Equations~{(\ref{H2R})} and {(\ref{homology-boundary})}, we rewrite
\begin{equation*}
i_{\ast }: H_{1}(\partial R - V_{0};\mathbb{Z}) \rightarrow H_{1}(R;
\mathbb{Z})
\end{equation*}
as
\begin{equation*}
i_{\ast }: \bigl ((H_{1}(V - \beta ; \mathbb{Z})) / ([\beta _{1}], [
\beta _{2}])\bigr )^{p}\oplus \mathbb{Z}^{p+1} \rightarrow (H_{1}(V;
\mathbb{Z})/[\beta ])^{\frac{p+1}{2}}\oplus \mathbb{Z}.
\end{equation*}

Note that $i_{\ast }$ maps the copy of $(H_{1}(V - \beta ; \mathbb{Z}))
/ ( [\beta _{1}], [\beta _{2}])$ coming from $V_{0}^{+}$ isomorphically
onto its image, and it ``pairs off'' the remaining $p-1$ copies of
$(H_{1}(V -\nobreak  \beta ; \mathbb{Z})) /\allowbreak  ( [\beta _{1}], [\beta _{2}])$,
including each of them into one of the remaining $\frac{p-1}{2}$ copies
of $H_{1}(V;\mathbb{Z})/[\beta ]$ in $R$. This contributes $H_{1}(V -
\beta ; \mathbb{Z})^{\frac{p-1}{2}}$ to $\ker (i_{\ast })$. In addition,
the generators for the $\mathbb{Z}^{p+1}$ summand in $H_{1}( \partial
R - V_{0};\mathbb{Z})$ can be chosen as follows. There are $
\mathbb{Z}^{\frac{p+1}{2}}$ generators which correspond to copies of
$\beta $ lying in the various copies of $V-\beta \subset (\partial R -
V_{0})$; they all map to the single $[\beta ]$ in the image,
contributing $\mathbb{Z}^{\frac{p-1}{2}}$ to $\ker i_{\ast }$. Finally,
there are an additional $\mathbb{Z}^{\frac{p+1}{2}}$ generators of the
$\mathbb{Z}^{p+1}$ summand in $H_{1}( \partial R - V_{0};\mathbb{Z})$
which are mapped injectively by $i_{\ast }$, onto classes in
$T/\bar{h}$ which are not in the image of $i_{\ast }(H_{1}(V - \beta
; \mathbb{Z})$ for any copy of $V$. Consequently, as we claimed,
\begin{equation*}
\ker (i_{\ast }) \cong (H_{1}(V - \beta ; \mathbb{Z})/([\beta _{1}], [
\beta _{2}]))^{\frac{p-1}{2}} \oplus \mathbb{Z}^{\frac{p-1}{2}}.\qedhere
\end{equation*}
\end{proof}

Furthermore, the above argument allows us to describe a basis for this kernel.

\begin{cor}
\label{basis}
Assume the notation of {Proposition~\ref{prop:homologyW}}. Further, let
$w^{1}, w^{2}, ... , w^{r}$ be a basis for $H_{1} (V-\beta ;
\mathbb{Z})/([\beta _{1}], [\beta _{2}])$, where $r=2g-2$ and $g$ is
the genus of $V$. Denote by $w^{i, \pm }_{j}, i\in \{1, ..., r\}, j
\in \{1, ..., p\}$ the pre-images of the $w^{i}$ lying in $f^{-1}(V
\times \{\pm 1\})$ so that $w^{i,\pm }_{j} \subset V^{\pm }_{j}$,
$f(w^{i,\pm }_{j})=w^{i}_{j}$ and $\tau w_{k}^{\pm }= w_{k+1\text{mod~} p}
^{\pm }$. Next, let $\bar{h}(w^{i,\pm }_{j})=:\overline{w}^{i, \pm }
_{j}\in M$. Finally, denote by $\overline{\beta }_{k}, k=0,1, ...,
\frac{p-1}{2}$ the $\frac{p+1}{2}$ generators of $H_{1}(\partial R - V
_{0};\mathbb{Z})$ which are represented by copies of $\beta $. Then, a
basis for $\ker i_{\ast }$ is given by:
%
\begin{equation}
\label{generators}
\{[\overline{w}^{i, +}_{k} - \overline{w}^{i, -}_{k}], [\overline{
\beta }_{k} -\overline{\beta }_{k-1}]\}_{i=1, ..., r; ~k= 1, ..., 
\frac{p-1}{2}}
\end{equation}

\end{cor}

\begin{proof}
The statement follows from the last paragraph of the proof of
{Proposition~\ref{prop:homologyW}}.
\end{proof}

\begin{rem}
The heavy notation we have had to resort to here deserves a comment.
Since $\bar{h}$ identifies $V_{j}^{+}$ with $V_{p-j}^{-}$, we have
$\overline{w}^{i,\pm }_{j}=\overline{w}^{i,\mp }_{p-j}$. Secondly, there
are many choices of $\frac{p+1}{2}$ curves $\beta _{k}$ so that the
classes $[\beta _{k}]$ are independent generators of $H_{1}( \partial
R - V_{0};\mathbb{Z})$. We note that it is possible to impose the extra
condition that $\overline{\beta }_{k}-\overline{\beta }_{k-1}$, $k=1, 2,
..., \frac{p-1}{2}$ bounds a cylinder $\beta \times [-1, 1]$ in $R$. We
do this by choosing for the $\overline{\beta }_{k}$-s ``consecutive''
copies of~$\beta $ as we move counter-clockwise in $\partial R$,
starting, for instance, with the copy of~$\beta _{1}$ lying
in~$V_{0}^{+}$. This observation will allow us to simplify the proof of
{Proposition~\ref{prop:signatureW}}.
\end{rem}

We are now ready to give a formula for the signature of $W$.

\begin{prop}
\label{prop:signatureW}
Let $\alpha $ be a knot which admits an irregular dihedral $p$-fold
cover for an odd prime $p$. In addition, assume that this cover is
$S^{3}$. Using the notation of {Corollary~\ref{basis}}, let $A$ be the
matrix of linking numbers in $S^{3}$ of the following set of links:
%
\begin{equation}
\label{basis-kernel}
S:=\{[\overline{w}^{i, +}_{k} - \overline{w}^{i, -}_{k}], [\overline{
\beta }_{k} -\overline{\beta }_{k-1}]\}_{i=1, ..., r; ~k= 1, ...,
\frac{p-1}{2}},
\end{equation}
where the orientation of each curve is compatible with a chosen
orientation on the corresponding curve, $w^{i}$ or $\beta $, in $V$.
Then, $\sigma (W) = \sigma (A)$.

\end{prop}

\begin{proof}
We wish to compute the intersection form on the image $i_{\ast }(H
_{2}(W; \mathbb{Z}))$ in $H_{2}(W, S^{3}\cup \Sigma ; \mathbb{Z})$.
Since $p$ is prime, $\Sigma $ is a rational homology
sphere~\cite{rolfsen1976knots}. It follows that
\begin{equation*}
H_{2}(W, M; \mathbb{Z})\cong i_{\ast }(H_{2}(W; \mathbb{Z}))\subset H
_{2}(W, S^{3}\cup \Sigma ; \mathbb{Z}).
\end{equation*}
By {Proposition~\ref{prop:homologyW}} we already know that
\begin{equation*}
H_{2}(W, M; \mathbb{Z})\cong \ker (i_{\ast }: H_{1}( \partial R - V
_{0};\mathbb{Z}) \rightarrow H_{1}(R;\mathbb{Z})) =: K.
\end{equation*}
Furthermore, by {Corollary~\ref{basis}}, the set of classes in
$S\subset H_{1}( \partial R - V_{0};\mathbb{Z})$ defined above forms a
basis for $K$. We use the isomorphism $H_{2}(W, M; \mathbb{Z})\cong K$
to obtain an explicit basis for $H_{2}(W, M; \mathbb{Z})$ consisting of
surfaces with boundary which are properly embedded in $W$.

Recall that the isomorphism between $H_{2}(W, M; \mathbb{Z})$ and
$K$ is given by Equation~{(\ref{eq:5})}, together with the boundary map
$\partial $ in the long exact sequence~{(\ref{LES-pair})}. By our choice of
basis for $K$, for any element $u\in S$, $u$ is the boundary of a
cylinder $S^{1}\times I =: U$ properly embedded in $(R, \partial R- V
_{0})\subset (W, M)$. We use this cylinder to represent the class
$[U]\in H_{2}(W, M; \mathbb{Z})$ corresponding to $u$ under the above
isomorphism $H_{2}(W, M; \mathbb{Z})\cong K$. Next, given $u_{1}, u
_{2}\in S$, we can write $u_{1} = a_{1} - a_{2}$ and $u_{2} = b_{1} -
b_{2}$, where $a_{i}, b_{i}$ are oriented curves in the dihedral cover
$M$. Denote by $U_{i}, i=1, 2$, the cylinder $S^{1}\times I \subset V
_{j}\times I$ with $\partial U_{i}=u_{i}$. (Here $V_{j}$ denotes the
lift of $V$ for which $a_{1}-a_{2}$, respectively $b_{1}-b_{2}$, lies
on the boundary of $V_{j}\times I$.) Now, if $F_{i}$ is any Seifert
surface for $u_{i}$ in $M\cong S^{3}$, we have $U_{i}\cup _{u_{i}} F
_{i}\in H_{2}(W; \mathbb{Z})$ and $i_{\ast }: H_{2}(W; \mathbb{Z})
\to H_{2}(W, S^{3}\cup \Sigma ; \mathbb{Z})$ carries $U_{i}\cup _{u
_{i}} F_{i}$ to $U_{i}$. We fix two Seifert surfaces, $F_{1}$ and
$F_{2}$, for $u_1$ and $u_2$, respectively, and use the classes $U_{1}\cup _{u_{1}} F_{1}$ and $U_{2}
\cup _{u_{2}} F_{2}$ to compute the intersection number $U_{1}\centerdot
U_{2}$.

By giving $W$ a little collar, $M\times [0, \epsilon ]$, and ``pushing 
in'' $U_{2}\cup _{u_{2}} F_{2}$ ever so slightly, we can assume that
$F_{2}$ lies in $M\times \{\epsilon \}$, and $U_{2}\cup _{u_{2}} F_{2}$
is disjoint from $M \times [0, \epsilon )$. Since $F_{1}\subset M
\times \{0\}$, in order to compute $U_{1}\centerdot U_{2}$, it now
suffices to consider the intersection of~$U_{1}$ with $U_{2}\cup _{u
_{2}} F_{2}$. We consider several cases. If the curves $a_{1}$ and
$b_{1}$ are disjoint, then so are $U_{1}=a_{1}\times I$ and
$U_{2}=b_{2}\times I$, regardless of whether the $a_{i}$, $b_{i}$ live
in the same lift of~$V_{j}$ or in different lifts. In this case, the
intersection is simply $U_{1}\centerdot F_{2} = lk (a_{1}-a_{2}, b
_{1}-b_{2})$, since $F_{2}\subset M\times \{\epsilon \}$, $U_{1}
\cap (M\times \{\epsilon \})= a_{1}-a_{2}$ and $F_{2}$ is a Seifert 
surface for $b_{1}-b_{2}$, so that, putting everything together, we have
$U_{1}\centerdot F_{2} = (a_{1}-a_{2})\cap F_{2} = lk (a_{1}-a_{2}, b
_{1}-b_{2})$ by definition. Secondly, $U_{1}$ and $U_{2}$ can be
distinct but intersecting cylinders. This can only happen if both live
in the same lift of $V\times I$, which we again denote $V_{j}\times I$.
In this case, we use the normal to $V_{j}\times I$ in $W$ to push off
$U_{1}$ away from $V_{j}\times I$ and thus from $U_{2}$. Again, we find
that $U_{1}\centerdot U_{2} = U_{1}\centerdot F_{2} = lk (a_{1}-a_{2},
b_{1}-b_{2})$. Lastly, we consider the case where $U_{1}=U_{2}$. For
some choice of $j$ we have $U_{1}\subset V_{j}\times I$ with
$\partial U_{1} = (a_{1} - a_{2}) \subset V_{j}\times \{0, 1\}$. We can
push $a_{1}$ off itself using its (positive, say) normal in
$V_{j}\times \{0\}$. This push-off extends across $U_{1} = (a_{1}
\times I) \subset (V_{j}\times I)$, so the cylinder can be made disjoint
from itself. Again, we conclude that $U_{1}\centerdot U_{2}= U_{1}
\centerdot F_{2} = lk (a_{1}-a_{2}, a_{1}-a_{2})$, where the
self-linking number is computed using the normal to $a_{1}-a_{2}$ in
$V_{j}\times \{0, 1\}$. Therefore, the matrix of linking numbers between
elements of our basis for $K$ is also the intersection matrix for
$W=(W(\alpha , \beta ))$. This completes the proof.
\end{proof}

\begin{rem}
We note that the self-linking with respect to the normal to
$a_{1}-a_{2}$ in $V_{j}\times \{0, 1\}$ is equal to the self-linking
with respect to the restriction to $a_{1}-a_{2}$ of the normal to
$V_{j}\times \{0, 1\}$ in $M\cong S^{3}$, since the two vectors are
everywhere linearly independent. This is useful for computations, since
the normal to $V_{j}$ in the dihedral cover is just the lift of the
normal to $V$ in $S_{3}$.
\end{rem}

The Proof of {Proposition~\ref{prop:homologyW}} also allows us to compute
the fundamental group of the manifold $W(\alpha , \beta )$ for knots
$\alpha $ which can arise as singularities of dihedral branched covers
between four-manifolds.

\begin{cor}
\label{pi1W}
Let $p$ be an odd prime and let $\alpha $ be a knot which admits a
$p$-fold irregular dihedral cover. Assume moreover that this cover
homeomorphic to $S^{3}$. Let $\beta $ be a characteristic knot for
$\alpha $ and let $W(\alpha , \beta )$ be the cobordism between
$S^{3}$ and the $p$-fold cyclic cover of $\beta $ constructed
in~\cite{CS1984linking}. Then $W(\alpha , \beta )$ is
simply-connected.
\end{cor}
\begin{proof}
We assume the notation of the proof of {Proposition~\ref{prop:homologyW}}.
(In this notation, the additional assumption of this Corollary is that
$M\cong S^{3}$.) We have seen that $W(\alpha , \beta )$ is homotopy
equivalent to $M\cup R$ and that $M\cap R= \partial R - V_{0}$. We also
know that $i_{\ast }: \pi _{1}(\partial R - V_{0}, a_{0}) \to \pi _{1}(R,
a_{0})$ is surjective. On the other hand, any loop in $ \pi _{1}(
\partial R - V_{0}, a_{0}) = \pi _{1}(M\cap R, a_{0})$ is contractible
in $M$ since $\pi _{1}(M; a_{0}) = 0$. Therefore, by van Kampen's
Theorem, $\pi _{1}(M\cup R, a_{0}) = 0 = \pi _{1}(W(\alpha , \beta ), a
_{0})$.
\end{proof}

Finally, we show that the defect to the signature of a branched cover
arising from the presence of a singularity $\alpha $ is an invariant of
the knot type $\alpha $.

\begin{prop}\label{prop:invariant}
Let $p$ be an odd square-free integer, and let $\alpha \subset S^{3}$
be knot which arises as the singularity of an irregular dihedral
$p$-fold cover between four-manifolds. Assume that $p^{2}$ does
not\footnote{One could allow $p^{2}$ to divide $\Delta (-1)$. In this
case, $\Xi _{p}$ would not necessarily be an invariant of the knot type
$\alpha $ but, rather, of $\alpha $ together with a specified
representation of $\pi _{1}(S^{3}-\alpha , x_{0}) \twoheadrightarrow D
_{p}$.} divide $\Delta (-1)$, where $\Delta (t)$ is the Alexander
polynomial of $\alpha $. In the notation of {Theorem~\ref{thm:necessary}},
the integer $ \Xi _{p}(\alpha )$, defined as
%
\begin{equation}
\label{valueXi}
\Xi _{p}(\alpha )= \frac{p^{2}-1}{6p}L_{V}(\beta , \beta ) + \sigma (W(
\alpha , \beta )) + \sum _{i=1}^{p-1} \sigma _{\zeta ^{i}}(\beta )
\end{equation}
is an invariant of the knot type $\alpha $.
\end{prop}

\begin{proof}
Since $\alpha $ arises as a singularity of an irregular dihedral
$p$-fold cover, by {Theorem~\ref{thm:necessary}}, $\alpha $ itself admits
an irregular dihedral $p$-fold cover. Since $p^{2}$ does not divide
$\Delta (-1)$, this cover is unique (see footnote on p.~166
of~\cite{CS1984linking} or, for a more thorough discussion,~\cite{fox1970metacyclic}).

When both $\alpha $ and $\beta $ are fixed, it is clear that each of the
terms $\frac{p^{2}-1}{6p} L_{V}(\beta , \beta )$, $\sigma (W(\alpha ,
\beta ))$ and $\sum _{i=1}^{p-1} \sigma _{\zeta ^{i}}(\beta )$ is
well-defined. We will show that their sum is in fact independent
of the choice of $\beta $.

Let $f: Y\to X$ be an irregular dihedral $p$-fold cover, branched over
an oriented surface $B\subset X$, embedded in $X$ with a unique
singularity of type $\alpha $. Such a cover exists by assumption. Then
\begin{equation*}
\Xi _{p}(\alpha) = p\sigma (X) -\frac{p-1}{2}e(B) - \sigma (Y),
\end{equation*}
a formula independent of the choice of $\beta $.

A priori, however, it might be possible for another branched cover
$f': Y'\to X'$, whose branching set also has a singularity of type
$\alpha $, to produce a different value of $\Xi _{p}$. This does not
occur. By the proof of {Theorem~\ref{thm:necessary}}, any choice of
characteristic knot $\beta $ can be used to compute the defect
$\Xi _{p}(\alpha )$ to the signature of $Y$. Using the same $\beta $ and
Equation~{(\ref{valueXi})} to compute this signature defect for two
different covers, for instance $Y$ and $Y'$, shows that $\Xi _{p}(
\alpha )$ does not vary with the choice of branched cover and indeed
depends only on $\alpha $.
\end{proof}

\section{Constructing dihedral covers}%
\label{sec:sufficient}
In this section, we describe a method for constructing an irregular
$p$-fold dihedral cover of a simply-connected four-manifold
$X$. We use this construction to prove {Theorem~\ref{thm:sufficient}},
which is a partial converse to {Theorem~\ref{thm:necessary}}. 
Precisely,
{Theorem~\ref{thm:sufficient}} establishes that, when two-bridge slice
singularities are considered, all pairs of integers $(\sigma , \chi )$ afforded by the
necessary condition ({Theorem~\ref{thm:necessary}}) as the signature and
Euler characteristic of a $p$-fold irregular dihedral cover of a given
base manifold $X$ with specified branching set $B$ are indeed realized
as the signature and Euler characteristic of a $p$-fold irregular
dihedral cover over $X$.

A main ingredient of the proof is constructing an irregular dihedral
cover of $S^{4}$ branched over a singular two-sphere with a given
singularity ({Proposition~\ref{covers_of_spheres}}). By taking a connected
sum with this singular two-sphere, we can introduce a singularity to
a PL embedded surface $B\subset X$ without changing its homeomorphism type
or that of the ambient manifold. The dihedral cover of $X$ is
constructed from the irregular dihedral cover of $(S^{4}, S^{2})$,
together with several copies of the double branched cover of $X$ over
the locally flat surface $B$.

We begin with a simple lemma.

\begin{lem}
\label{lem:SC_double_cover_exists}
Let $X$ be four-manifold and let $B\subset X$ be an embedded connected
surface such that $\pi _{1}(X-B, x_{0})\cong \mathbb{Z}/2\mathbb{Z}$. The
double branched cover of $(X, B)$ is simply-connected.
\end{lem}

\begin{proof}
Since $\pi _{1}(X- B, x_{0})\cong \mathbb{Z}/2\mathbb{Z}$, a double cover
of $X$ branched along $B$ exists. We denote this cover by $\hat{X}$ and denote by $\hat{B}$ the (homeomorphic) pre-image of $B$ under the
covering map. We apply van Kampen's theorem to $\hat{X} = (\hat{X} -
\hat{B})\cup _{\partial N(\hat{B})} N(\hat{B})$, where $N(\hat{B})$
denotes a small tubular neighborhood of $\hat{B}$. Being the universal
cover of $(X- B)$, $(\hat{X} - \hat{B})$ is simply connected, so
$i_{\ast }: \pi _{1}(\partial N(\hat{B}), b_{0})\rightarrow \pi _{1}(
\hat{X}-\hat{B}, b_{0})$ is the zero homomorphism. In addition, $i_{\ast }: \pi
_{1}(\partial N(\hat{B}), b_{0})\rightarrow \pi _{1}(N(\hat{B}), b_{0})$
is surjective, since every element in $\pi _{1}(N(\hat{B}), b_{0})$ can
be represented by a loop which is disjoint from the 0-section and which
is therefore homotopic to a loop in $\partial N(\hat{B})$. It follows
from van Kampen's Theorem that $\hat{X}$ is simply-connected.
\end{proof}

Next, we prove a couple of lemmas concerning the singularities which we
will be introduced to the branching sets in the construction of
dihedral covers. In {Lemma~\ref{ribbon_disks}} we recall a well-known fact
about the fundamental groups of complements of ribbon disks.
{Lemma~\ref{fox-milnor}} allows us to extend a dihedral cover of a
two-bridge slice knot to a cover of a disk it bounds.

\begin{lem}
\label{ribbon_disks}
Let $K\subset S^{3} =\partial B^{4}$ be a ribbon knot and let $D'\subset S^{3}$ be a
ribbon disk for $K$. Then, there exists $D\subset B^{4}$, a slice disk
for $K$, such that the map $i_{\ast }: \pi _{1} (S^{3}-K, x_{0})
\to \pi _{1}(B^{4}-D, x_{0})$ induced by inclusion is surjective.
\end{lem}
\begin{proof}
Since $D'$ is ribbon, we can push the interior of $D'$ into the interior
of $B^{4}$ to obtain a slice disk $D$ with the property that $g$, the
radial function on $B^{4}$, is Morse when restricted to $D$ and has no
local maxima on the interior of $D$. Computing the fundamental group of
the complement of $D$ in $B^{4}$ by cross-sections
as outlined in~\cite{fox1962quick}, we start with $\pi _{1} (\partial B^{4}-
\partial D, x_{0}) = \pi _{1}(S^{3}- K, x_{0})$ and proceed to introduce
new generators or relations at each critical point of $g$. Since $g$ has
no maxima, no new generators are introduced, implying that $i_{\ast }:
\pi _{1}(S^{3}- K, x_{0})\rightarrow \pi _{1} (B^{4}-D, x_{0})$ is a
surjection.
\end{proof}

In the notation of the above lemma, a disk $D$ with the property that
\begin{equation*}
i_{\ast }: \pi _{1} (S^{3}-K, x_{0}) \to \pi _{1}(B^{4}-D, x_{0})
\end{equation*}
is a surjection is called a \emph{homotopy ribbon disk}. Thus, the lemma can
be rephrased by saying ribbon knots admit homotopy ribbon disks.

\begin{lem}
\label{fox-milnor}
Let $K\subset S^{3}\subset \partial B^{4}$ be a slice knot and let
$D\subset B^{4}$ be a slice disk for $K$. Let $p>1$ be an odd
square-free integer. If the pair $(S^{3}, K)$ admits an irregular $p$-fold
dihedral cover, then the pair $(B^{4}, D)$ admits one as well.
Furthermore, if $K$ is a two-bridge knot, $D$ can be chosen PL and such
that the irregular dihedral cover of $B^{4}$ branched along $D$ is
simply-connected.
\end{lem}

\begin{proof}
Let $\Delta _{K}(t)$ denote the Alexander polynomial of $K$ and
$\Delta _{D}(t)$ that of $D$. Denote by $\hat{S}$ the double branched
cover of the pair $(S^{3}, K)$ and by $\hat{K}$ the pre-image of $K$
under the covering map. It is well known that $|\Delta _{K}(-1)| = |(H
_{1} (\hat{S}; \mathbb{Z})|$ \cite{seifert1935geschlecht}. Similarly,
denote by $\hat{B}$ the double cover of $B^{4}$ branched along $D$ and
by $\hat{D}$ the pre-image of $D$. As above, we have $|\Delta _{D}(-1)|=
|(H_{1} (\hat{B}; \mathbb{Z})|$, since $\pm \Delta _{D}(-1)$ is the
determinant of a presentation matrix for the first homology of the
double branched cover of $D$. (Denote by $B_{\infty }$ the infinite
cyclic cover of the disk complement $B^{4}-D$. Regard $H_{1}(B_{
\infty }; \mathbb{Z})$ as a $\mathbb{Z}[\tau , \tau ^{-1}]$-module, where
the action of $\tau $ is that induced by a generator of the group of
covering translations. Then $\Delta _{D}(t)$ is the characteristic
polynomial of this action and $H_{1}(\hat{B}; \mathbb{Z})\cong Coker
\{1-\tau ^{2}: H_{1}(B_{\infty }; \mathbb{Z})\to H_{1}(B_{\infty };
\mathbb{Z})\}$. For a thorough exposition on the homology of cyclic
covers of a homology $S^{1}$, see~\cite{stevens1996homology}.)

Since $K$ admits a dihedral cover, $H_{1} (\hat{S}; \mathbb{Z})$
surjects onto $\mathbb{Z}/p\mathbb{Z}$~\cite{CS1984linking}. It
follows that $\Delta _{K}(-1)\equiv 0 \mod p$. Since $D$ is a slice disk
for $K$, by results of Fox and Milnor \cite{fox1966singularities} we
have $\Delta _{K}(-1) = \pm (\Delta _{D}(-1))^{2}$, so $(\Delta _{D}(-1))^{2}
\equiv 0 \mod p$. Since $p$ is square-free by assumption, we conclude that
$\Delta _{D}(-1) \equiv 0 \mod p$ as well. Then $H_{1} (\hat{B};
\mathbb{Z})$ surjects onto $\mathbb{Z}/p\mathbb{Z}$, and thus
$\hat{B}$ admits a $p$-fold cyclic cover $T$ with $\partial T = : N$.
This cover $T$ is the {regular} dihedral $2p$-fold branched cover
of $(B^{4}, D)$. Let $Z$ be the quotient of $T$ by the action of any
$\mathbb{Z}/2\mathbb{Z}$ subgroup of $D_{p}$. Then $Z$ is the desired
irregular dihedral $p$-fold cover of $(B^{4}, D)$. Its boundary, which
we denote by $U$, is the irregular dihedral $p$-fold cover of $K$.

Now assume in addition that $K$ is a two-bridge knot. In this case it
is well-known that $U$ is in fact $S^{3}$. Indeed, the pre-image
$S^{\ast }$ of a bridge sphere for $K$ is a dihedral cover of~$S^{2}$ branched over four points, so $S^{\ast }$ has Euler
characteristic
\begin{equation*}
\chi (S^{\ast }) = p (\chi (S^{2}) - 4) + 4\frac{p+1}{2} = 2.
\end{equation*}
A bridge sphere bounds a trivial tangle to either side, and the cover
of a trivial tangle is a handlebody. Therefore, $S^{\ast }$ is a
Heegaard surface for $U$, and, since the genus of $S^{\ast }$ is zero,
$U\cong S^{3}$.

Since $K$ is two-bridge slice, it is ribbon. Hence, by
{Lemma~\ref{ribbon_disks}}, the slice disk $D$ for $K$ can be chosen to
be PL and homotopy ribbon, i.e. such that $\pi _{1} (S^{3}-K, x_{0})\xrightarrow{i
_{\ast }} \pi _{1}(B^{4}-D, x_{0})$ is a surjection. Therefore, given a
homomorphism $\psi : \pi _{1} (B^{4}-D, x_{0}) \rightarrow D_{2p}$, the
pre-image $(\psi \circ i_{\ast })^{-1}(\mathbb{Z}/2\mathbb{Z})$ surjects
onto $\psi ^{-1}(\mathbb{Z}/2\mathbb{Z})$ by $i_{\ast }$. This implies
that the inclusion of the unbranched cover associated to $U$ into the
unbranched cover associated to $Z$ induces a surjection on fundamental
groups. Since the branching set of $U$ is a subset of the branching set
of $Z$, it follows that $\pi _{1}(U, x_{0})\xrightarrow{i_{\ast }}
\pi _{1}(Z, x_{0})$ is also a surjection. But $\pi _{1}(U, x_{0}) = 0$,
and we conclude that the irregular dihedral cover of the pair
$(B^{4}, D)$ is simply-connected.
\end{proof}

\begin{prop}
\label{covers_of_spheres}
Let $p>1$ be an odd square-free integer and let $K\subset S^{3}$ be a
slice knot such that the pair $(S^{3}, K)$ admits an irregular dihedral $p$-fold
 cover. Then there exists an embedded two-sphere $S^{2}\subset
S^{4}$ such that the pair $(S^{4}, S^{2})$ admits an irregular dihedral $p$-fold
 cover $W$ and $S^{2}\subset S^{4}$ is locally flat except at
one point where it has a singularity of type $K$. Moreover, if $K$ a
two-bridge knot, $W$ is a simply-connected manifold.
\end{prop}

\begin{proof}
Let $D_{1}^{2}\subset B^{4}_{1}$ be a PL slice disk for $K$. Denote the
cone on the pair $(S^{3}, K)$ by $(B_{2}^{4}, D_{2}^{2})$. The disk $D_{2}^{2}$ is a PL submanifold of $B^{4}_{2}$
except at the cone point $x$, where by construction $D_{2}^{2}$ has a
singularity of type $K$. Identifying the two pairs $(B_{1}^{4}, D_{1}
^{2})$ and $(B_{2}^{4}, D_{2}^{2})$ via the identity map along the two
copies of $(S^{3}, K)$ lying on their boundaries, we obtain an embedding
of a two-sphere $S: = D_{1}^{2} \cup _{K} D_{2}^{2}$ in $S^{4} = B_{1}
^{4} \cup _{S^{3}} B_{2}^{4}$ such that $S$ has a unique
singularity of type $K$ at $x$.

By {Lemma~\ref{fox-milnor}}, the pair $(B^{4}_{1}, D_{1}^{2})$ admits an
irregular dihedral $p$-fold cover $W$, and its boundary $M$ is the
irregular dihedral $p$-fold cover of the pair $(S^{3}, K)$. Since
$(B^{4}_{2}, D_{2}^{2})$ is a cone, its irregular dihedral $p$-fold
cover is simply the cone on $M$. Thus, the pair $(S^{4}, S)$ admits a
cover
\begin{equation*}
Z:=W\bigcup _{\partial W\sim M\times \{0\}} (M\times [0, 1]/ M\times
\{1\})
\end{equation*}
as claimed. If, in addition, $K$ is a two-bridge knot, by
{Lemma~\ref{fox-milnor}} we know that $M$ is the three-sphere and 
that we can pick the disk $D_{1}^{2}$ in the above construction to be homotopy ribbon so that $W$ is
simply-connected. Thus, for $K$ two-bridge, $Z$ is a simply-connected
manifold.
\end{proof}

\begin{proof}[Proof of {Theorem~\ref{thm:sufficient}}]
The proof is as follows: first, we modify the branching set $B$ by
introducing a singularity of type $\alpha $ to the embedding of
$B$ in $X$; next, we construct the desired covering space $Y$ by gluing
together several manifolds by homeomorphisms on their boundaries; we
check that $Y$ is indeed a $p$-fold irregular dihedral cover of $X$ with
the specified branching set; lastly, we verify that $Y$ is a
simply-connected manifold.

We begin by modifying the branching set as outlined above. Let
$S^{2}\subset S^{4}$ be an embedded two-sphere with a unique singularity
of type $\alpha $, constructed as in the proof of {Proposition~\ref{covers_of_spheres}}.
Let $y\in S^{2} \subset S^{4}$ be any locally flat point with
$N(y)$ a neighborhood of $y$ not containing the singular point $x$. We
use $N(y)$ to form the connected sum of pairs $(X, B)\# (S^{4}, S^{2})
= : (X, B_{1})$. By construction, $B_{1}$ is homeomorphic to $B$, is
embedded in $X$ with a unique singularity of type $\alpha $ and
satisfies $e(B_{1})=e(B)$. Furthermore, we see from the natural
Mayer--Vietoris sequence that $H_{1} (X- B; \mathbb{Z})\cong H_{1} (X-
B_{1}; \mathbb{Z})$ and the latter group is $\mathbb{Z}/2\mathbb{Z}$ by
assumption. Hence, $X$ admits a double cover $f: \hat{X}\rightarrow X$
branched along $B_{1}$.

A useful way to visualize this cover is the following. Since
$y\in S^{4}$ is a locally flat point, $B\cap \partial N(y)$ is the
unknot. Now viewing $\partial N(y)$ as embedded in $B_{1}$, we note that
the restriction of $f$ to $f^{-1}(\partial N(y))$ is a double branched
cover of the trivial knot, whose total space is again $S^{3}$.
Furthermore, the pre-images under $f$ of the connected summands
$(X, B) - B^{4}$ and $(S^{4}, S^{2})-N(y)$ are the double branched
covers of those summands. We can thus think of the double branched cover
$\hat{X}$ of the pair $(X, B_{1})$ as the union (along $S^{3}$ viewed
as a double cover of the unknot) of the double branched covers of a
punctured $(X, B)$ and $(S^{4}, S^{2})-N(y)$. For future use, we denote
the restriction of $f$ to the pre-image $\hat{X}_{0}$ of $X- N(x)$ by
$f_{0}$,
\begin{equation*}
f_{0}:\hat{X}_{0}\rightarrow (X- N(x)).
\end{equation*}

Next, consider the irregular dihedral $p$-fold cover $g: Z\rightarrow
S^{4}$ of $(S^{4}, S^{2})$ constructed as in
{Proposition~\ref{covers_of_spheres}}. For $y$ as above, the restriction
of $g$ to $g^{-1}(\partial N(y))$ is the irregular dihedral $p$-fold
cover of the unknot, which consists of the disjoint union of
$\frac{p+1}{2}$ copies of $S^{3}$, $\frac{p-1}{2}$ of which are double
covers and one a single cover. Furthermore, $g^{-1}(S^{4}-N(y))$ is the
irregular dihedral $p$-fold cover of the pair $(B^{4}, D^{2})$, where
the two-disk is singular. The boundary of this dihedral cover consists
of ${\frac{p+1}{2}}$ copies of $S^{3}$. Of those, $\frac{p-1}{2}$
double-cover the complement of the unknot and one is mapped
homeomorphically by $g$.

We now describe the manifold $Y$ which we will show is homeomorphic to
a dihedral cover of $X$ along $B_{1}$. We attach to $g^{-1}(S^{4}-N(y))$
a copy of $\hat{X}_{0}$ along each boundary component $S^{3}$ which
double-covers the complement of the unknot. Naturally, the attachment
identifies the boundary components by a homeomorphism of pairs
$(S^{3}, S^{1})$, where the second component is the (unknotted)
branching set. In the same manner, we also attach a punctured copy of
$X$ along the boundary component $S^{3}$ which is a cover of index
$1$. The map
\begin{equation*}
h:= g\cup _{\frac{p-1}{2}}f_{0}\cup 1_{X-N(\ast )} : Y\rightarrow X
\end{equation*}
is a branched cover of $(X, B_{1})$. By construction, $h$ satisfies the
property that for all points $z\in B- x$, if $N(z)$ is a small
neighborhood of $z$ in $X$ not containing $x$, then $h^{-1}(N(z))$ has
$\frac{p-1}{2}$ components of index $2$ and one component of index
$1$. So $Y$ is the desired dihedral cover. By
{Theorem~\ref{thm:necessary}}, the Euler characteristic and signature of
$Y$ are those determined by the prescribed triple $X, B, \alpha $.
(Here, we use the fact that the above construction does not change the
homeomorphism type or self-intersection number of~$B$.)

Finally, we observe that $Y$ consists of simply-connected manifolds
joined together via homeomorphisms on their boundaries. Indeed, $X$ is
simply-connected by assumption, and $\hat{X}$ is simply-connected by
{Proposition~\ref{lem:SC_double_cover_exists}}. The irregular dihedral
cover $Z$ of $S^{4}$ is simply-connected by
{Proposition~\ref{covers_of_spheres}}, and, therefore, so is $g^{-1}(S
^{4}-N(y))$. We conclude that $Y$ is simply-connected, which completes
the proof.
\end{proof}

\begin{rem}
One can obtain analogous results by varying the hypotheses on the branching set $B$.
For instance, if we do not require that our construction result in a
simply-connected cover, we can relax the condition that $\pi _{1}(X-B,
x_{0})\cong \mathbb{Z}/2\mathbb{Z}$ and use for our branching set any
surface $B$ which represents an even class in $H_{2}(X; \mathbb{Z})$.
This allows us to produce, by introducing any two-bridge slice
knot as the singularity and by varying the genus of $B$ (see
{Lemma~\ref{genus_increase}}), infinitely many dihedral branched covers
of $S^{4}$, which are easily distinguished by their Euler
characteristic. Furthermore, if $B$ is the boundary union of the cone on a two-bridge knot $\alpha$ and a homotopy ribbon surface for $\alpha$, one can construct simply-connected covers of $S^4$ by this method, as done in~\cite{cahn2017singular}.  
\end{rem}

We can also use the techniques of {Theorem~\ref{thm:sufficient}} to
construct over a given four-manifold $X$ an infinite family of dihedral
covers with the same singularity type on the branching set
({Theorem~\ref{thm:odd_indefinite}}). The first step is to establish the
following Lemma.

\begin{lem}
\label{genus_increase}
Let $B\subset X^{4}$ be an oriented surface of genus $g$, PL embedded in
$X$ and such that $\pi _{1}(X- B, x_{0})\cong \mathbb{Z}/2\mathbb{Z}$. Then,
there exists a PL embedded oriented surface $C$ of genus $g+1$ in
$X$ such that $\pi _{1}(X- C, x_{0})\cong \mathbb{Z}/2\mathbb{Z}$, and
such that $e(B) = e(C)$.
\end{lem}

\begin{proof}
Let $T\subset S^{4}$ be the standard embedding of the two-torus in the
four-sphere. We have $\pi _{1} (S^{4}- T, x_{0}) \cong \mathbb{Z}$,
generated by any meridian of $T$ in $S^{4}$.

Now consider the connected sum of pairs $(X, B) \# (S^{4}, T)$ and let
$C= B \#T \subset X\# S^{4} \cong X$. Note that the genus of $C$ is
one higher than that of $B$. Since a meridian $m_{1}$ of
$T$ in $S^{4}$ becomes identified under the connected sum with a
meridian $m_{2}$ of $B$ in $X$, it follows that the fundamental group
of $ (X- C)$ is isomorphic to $\langle m_{1}, m_{2} | m_{1}=m_{2}, m
_{2}^{2} = 0 \rangle \cong \mathbb{Z}/2\mathbb{Z}$.

Finally, under the isomorphism of pairs $(X, B) \# (S^{4}, T) \cong (X,
C)$, the class $[C]\in H_{2}(X; \mathbb{Z})$ corresponds to the class
$[B\#T] \in H_{2}(X\# S^{4}; \mathbb{Z})$. Since $[T] = 0\in H_{2}(X
\# S^{4}; \mathbb{Z})$, indeed $e(B) = e(C)$. 
\end{proof}

We have now done most of the work needed to obtain an infinite family
of covers over a given base.

\begin{proof}[Proof of {Theorem~\ref{thm:construction-slice}}]
The first step of our construction is to find a closed surface
\mbox{$B\subset X$}, PL embedded in $X$ and such that $\pi _{1} (X-B; x_{0})
\cong \mathbb{Z}/2\mathbb{Z}$. Since $X$ is simply-connected and its
second Betti number is positive, such a surface exists, as we now show.
Let $F$ be a closed oriented surface, smoothly embedded in $X$ and such
that the maximum divisibility of $[F]$ in $H_{2}(X; \mathbb{Z})$ is 2.
Then $H_{1}(X-F; \mathbb{Z})\cong \mathbb{Z}/2\mathbb{Z}$. By classical
techniques, $F$ can be modified to produce a new surface $F'$, carrying
the same homology class as $F$, the fundamental group of whose
complement is abelian, as follows. Since $X$ is simply-connected,
$\pi _{1}(X-F)$ is normally generated by a meridian $\mu $ of $F$. For
any $g\in \pi _{1}(X-B, x_{0})$, the commutator $[\mu , g\mu g^{-1}]$ can
be killed by performing a finger move on the surface $F$, as shown in
Lemma~1 of~\cite{casson1986three}. After iterating this move finitely
many times, the result is a self-transverse immersed surface $F'$,
the fundamental group of whose complement is generated by $\mu $.
Finally, self-intersections of $F'$ can be removed by replacing, in a
small neighborhood of any double point, the cone on the Hopf link by an
annulus. This operation has no effect on the fundamental group of the
complement and produces the desired surface $B$.

The next step is to further modify the surface to introduce a
singularity of the desired type. Following the procedure in the proof
of {Theorem~\ref{thm:sufficient}}, we use a two-sphere $S\subset S
^{4}$, PL embedded in $S^4$ except for one singularity of type $\alpha $; next,
we construct a $p$-fold irregular dihedral cover of the pair
$(X, B)\#(S^{4}, S) \cong (X, B)$, as in the proof of
{Theorem~\ref{thm:sufficient}}.

Fixing a knot $\alpha $ as the singularity type, by
{Lemma~\ref{genus_increase}}, we can increase the genus of the branching
set $B$ to obtain an infinite family of such covers. These covers are pairwise non-homeomorphic and can be
distinguished by their Euler characteristics. Using knots for which the values of
$\Xi _{p}$ differ, it is possible to obtain covers distinguished by their
signatures as well.
\end{proof}

As an immediate consequence of our construction, we have the following.

\begin{cor}
Let $(\sigma , \chi )$ be a pair of integers which satisfy
Equations~{(\ref{eq:chi})} and~{(\ref{eq:sigma})} for some given $X$, $B$,
$\alpha $ and $p$, where $p$ is an odd prime and $\alpha $ a two-bridge
slice knot. Then, if $\chi ' = \chi + (p-1)k$ for a natural number
$k$, there exists a manifold $Y'$ which is homeomorphic to a $p$-fold
irregular dihedral cover of $X$ and satisfies $\sigma (Y') =\sigma $,
$\chi (Y')=\chi '$. Moreover, if $\pi _{1}(X-B, x_{0})\cong \mathbb{Z}/2
\mathbb{Z}$, $Y$ is simply-connected.
\end{cor}

We conclude by proving {Theorem~\ref{thm:odd_indefinite}}.

\begin{proof}[Proof of {Theorem~\ref{thm:odd_indefinite}}]
$(=>)$ If $Y$ is homeomorphic to a dihedral $p$-fold cover of $X$ with
the specified branching data, by {Theorem~\ref{thm:necessary}}, the Euler
characteristic and signature of~$Y$ satisfy Equations~{(\ref{eq:chi})}
and~{(\ref{eq:sigma})} with respect to $B_{1}$ and thus, by assumption, with
respect to $B$.

$(<=)$ Assume the Euler characteristic and signature of $Y$ satisfy
Equations~{(\ref{eq:chi})} and~{(\ref{eq:sigma})}. We will construct a branched
cover of $X$ whose branching set has the specified properties, and we
will prove that this cover is homeomorphic to $Y$.

We follow the steps used in the proof of {Theorem~\ref{thm:sufficient}}
to construct a $p$-fold irregular dihedral cover of $X$ branched over
a surface $B_{1}\cong B$ which is embedded in $X$ with a singularity of
type $\alpha $ and so that $e(B_{1})=e(B)$. Call this cover $Z$. Since
$\alpha $ is a two-bridge slice knot, by {Theorem~\ref{thm:sufficient}}, $Z$ is a simply-connected manifold. We will prove that the
intersection form of $Z$ is equivalent to that of $Y$.

Being a dihedral cover of $X$, $Z$ satisfies the equations set forth in
{Theorem~\ref{thm:necessary}}, where, again, $B$ and $B_{1}$ can be used
interchangeably. By assumption, $Y$ also satisfies these equations,
so $\sigma (Y) = \sigma (Z)$ and $\chi (Y)=\chi (Z)$. Since $Y$ is a
simply-connected four-manifold, the rank of $H_{2}(Y; \mathbb{Z})$ is
$\chi (Y) - 2$, and the analogous statement holds for $Z$. In other words, the intersection forms of $Y$ and $Z$ have the same
signature and rank. The intersection form of $Y$ is odd by assumption.
The intersection form of $Z$ is also odd because by construction $Z$ has
a copy of $X$ as a connected summand and $X$ itself is odd. Therefore,
the intersection forms of $Y$ and $Z$ have the same signature, rank and
parity. In particular, both are definite or both are indefinite. If both
forms are definite, since they arise as intersection forms of smooth
four-manifolds, by Donaldson's result~\cite{donaldson1983application},
each diagonalizes to $\pm I_{n}$, where $n=\chi (Y) - 2=\chi (Z) - 2$
and the sign determined by $\sigma (Y)=\sigma (Z)$. If both are
indefinite, we again conclude that they are isomorphic, this time using
Serre's classification~\cite{serre1961formes} of indefinite unimodular
integral bilinear forms. By~Freedman's classification of
simply-connected four-manifolds~\cite{freedman1982topology}, it follows that $Y$ and $Z$ are homeomorphic.
\end{proof}

\section*{Appendix A. Characteristic knots}\label{appA}
Our construction of an infinite family of irregular dihedral $p$-fold
covers of over a given four-manifold
({Theorem~\ref{thm:construction-slice}}) hinges on being able to find
two-bridge slice knots which admit dihedral $p$-fold covers themselves. In this section we prove
that, for any odd prime $p$, infinitely many such knots exist. In
particular, we exhibit for every $p$ an infinite class of knots for
which the necessary condition ({Theorem~\ref{thm:necessary}}) for the
existence of a dihedral $p$-fold cover over a given base is sharp. As a
biproduct, we also illustrate how to find characteristic knots in the
two-bridge case.

Recall that Lisca~\cite{lisca2007lens} proved that, for two-bridge
knots, being slice is equivalent to being ribbon. Previously, Casson and
Gordon~\cite{cass-gor1986cobordism} gave a necessary condition for a
two-bridge knot to be ribbon, and Lamm~\cite{lamm2000symmetric}
\cite{lamm2006symmetric} listed all knots satisfying this condition.
He found that for all $a\neq 0, b\neq 0$ the knots $K_{1}(a,b) =
C(2a,2,2b,-2,-2a,2b)$ and $K_{2}(a,b) = C(2a, 2, 2b, 2a, 2, 2b)$ are
ribbon. {Fig.~\ref{fig:slice_knot}} recalls the notation $C(e_{1}, ... , e_{6})$. In
{Fig.~\ref{fig:seifert_surface}} we give a genus $3$ Seifert surface $V$ for the knot
$\alpha = C(e_{1},e_{2},e_{3},e_{4},e_{5},e_{6})$. We use the surface
$V$ for all subsequent computations.

Since two-bridge slice knots play a key role our construction of
dihedral covers of four-manifolds, we determine the values of the
parameters $a$ and $b$ for which the knots $K_{i}(a, b)$ admit
three-fold dihedral covers.

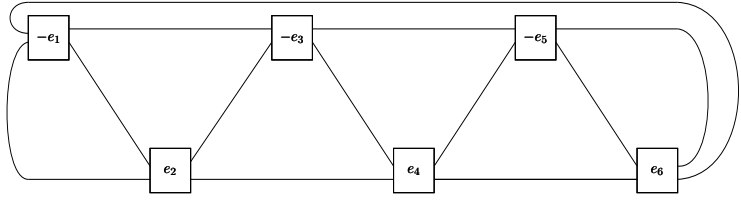
\begin{figure}
\resizebox{4in}{1in} {
\begin{tikzpicture}

\foreach \x in {1,3,5}
\foreach \y in {2,4,6} {
\draw (\x+\x+\x,3) +(-.5,-.5) rectangle ++(.5,.5); \draw (\x+\x+\x,3) node{$-e_{\x}$};
\draw (\y +\y+\y,0) +(-.5,-.5) rectangle ++(.5,.5); \draw (\y +\y+\y,0) node{$e_{\y}$};
}

\draw (3.5, 3.2)--(8.5,3.2); 
\foreach\z in {0, 6,12}  \draw [xshift=\z cm] (3.5, 2.9)--(5.5, 0.1); 
\draw (9.5, 3.2)--(14.5,3.2);
\draw (6.5, -0.2)--(11.5,-0.2);
\draw (12.5, -0.2)--(17.5,-0.2); 
\draw (12.5, -0.2)--(17.5,-0.2);
\draw (6.5, 0.2) -- (8.5, 2.9);/Users/sasheto/Desktop/branched covers/thesis stuff/twobridgeKNOT.tex
\draw[xshift=6cm] (6.5, 0.1) -- (8.5, 2.9);
\draw (2.5, 3.1) ..controls (1.9,3.1) and(1.9, 3.8)  .. (2.5, 3.8);
\draw (2.5, -0.2).. controls (1.8, -0.2) and (1.8,2.8) ..(2.5, 2.9);
\draw  (2.5, 3.8) -- (18.5, 3.8);
\draw (2.5, -0.2)--(5.5, -0.2);
\draw (18.5, 3.8)..controls (20.5, 3.7) and (20.5, -0.1).. (18.5,-0.2);
\draw (15.5,3.2)--(18.5, 3.2);
\draw (18.5, 3.2).. controls (19.5, 3.1) and (19.5,0).. (18.5, 0.1);

\end{tikzpicture}}
\caption{The knot $C(e_{1},...,e_{6})$. Each square represents a two-strand
braid with only positive or only negative horizontal twists, according to the
sign of $e_{i}$. The absolute value of $e_{i}$ gives the number of crossings.} \label{fig:slice_knot}
\vspace*{-2pt}
\end{figure}

\begin{figure}
\colorlet{darkred}{red!95!black}

\resizebox{5in}{1.5in} {
\begin{tikzpicture} 

\draw (0, 0.17) .. controls  (-2.5, 0.6) and (-2.5,3)     .. (0, 3.6); 
\draw (0, 3.6) .. controls (3.5, 4.5) and (10,4.5) ..  (13.5, 3.6);  
\draw (13.5, 0.17) .. controls  (16, 0.6)  and (16, 3)   .. (13.5, 3.6); 

\draw (1,0.17) .. controls (3,1.6) .. (5,0.17); 
\draw (6,0.17) .. controls (8,1.6) .. (10,0.17);  
\draw   (13.5, 0.53) .. controls (14.5, 1)  and  (13.6, 1.5) ..    (11, 0.17);
\draw (1.75,2.5) .. controls (3, 2.5) and (2 , 1)   .. (1, 0.53);

\begin{scope}[xshift=2.5cm]
 \draw   (2.5,0.53) .. controls (0.5,2) and (0.6,2.5) .. (1.75,2.5); 
\end{scope}

\begin{scope}[xshift=7.5cm]
 \draw   (2.5,0.53) .. controls (0.5,2) and (0.6,2.5) .. (1.75,2.5); 
\end{scope} 

\begin{scope}[xshift=10cm]
    \draw (1.75,2.5) .. controls (3, 2.5) and (2 , 1)   .. (1, 0.53);
\end{scope} 

\begin{scope}[xshift=5cm]
 \draw (1.75,2.5) .. controls (3, 2.5) and (2 , 1)   .. (1, 0.53);
\end{scope} 

\draw  (1.6,0.9) node [circle,  fill=white] {} ;
\draw (1.85,0.85)  node [circle, fill=white] {} ;
\draw (4.5,1) node [circle, fill=white]   {};
\draw (6.7,1) node [circle, fill=white]   {};
\draw  (9.3,1) node [circle, fill=white]  {};
\draw (11.6,1) node [rectangle, fill=white] {}  ;

\begin{scope} [xshift=-2.6cm, yshift=-0.33] 
       \begin{scope}[yscale=1,xscale=-1, xshift=-3cm]    
\draw[darkred,  very thick, densely dashed] [rotate=90] (0.15, 0) parabola bend (2.2, 2.2) (3 ,1) ; 
  \end{scope}  
   
 \draw [darkred, densely dashed]  [->, very thick] (2,3) .. controls (3, 3.2) .. (4,3)node [midway, above] {$\omega_1$}; 
    
 \begin{scope}[xshift=3cm]    
\draw[darkred,  very thick, densely dashed] [rotate=90] (0.15, 0) parabola bend (2.2, 2) (3 ,1) ;  
  \end{scope}    \end{scope}        
\draw (1.75,0.9) node [circle, fill=white] {} ;

\begin{scope} [xshift=2.5cm] 
       \begin{scope}[yscale=1,xscale=-1, xshift=-3cm]    
\draw[darkred,  very thick, densely dashed] [rotate=90] (0.15, 0) parabola bend (2.2, 2.2) (3 ,1) ; 
  \end{scope}  
   
 \draw [darkred, densely dashed]  [->, very thick] (2,3) .. controls (3, 3.2) .. (4,3)node [midway, above] {$\omega_3$};
    
 \begin{scope}[xshift=3cm]    
\draw[darkred,  very thick, densely dashed] [rotate=90] (0.15, 0) parabola bend (2.2, 2.2) (3 ,1) ;  
  \end{scope}    \end{scope}

    \begin{scope} [xshift=7.5cm] 
       \begin{scope}[yscale=1,xscale=-1, xshift=-3cm]    
\draw[darkred,  very thick, densely dashed] [rotate=90] (0.15, 0) parabola bend (2.2, 2.2) (3 ,1) ; 
  \end{scope}  
   
 \draw [darkred, densely dashed]  [->, very thick] (2,3) .. controls (3, 3.2) .. (4,3)node [midway, above] {$\omega_5$};
    
 \begin{scope}[xshift=3cm]    
\draw[darkred,  very thick, densely dashed] [rotate=90] (0.15, 0) parabola bend (2.2, 2.2) (3 ,1) ;  
  \end{scope}    \end{scope}

\begin{scope}[xshift=5cm]
 \draw  [ white , double=black , line width=4.44 , double distance =0.3pt]  (2.5,0.53) .. controls (0.5,2) and (0.6,2.5) .. (1.75,2.5); 
\end{scope}

\draw  [ white , double=black , line width=3.8 , double distance =0.3pt]  (0,0.53) .. controls (-1,1) and (0.6,1.5) .. (2.5,0.17);
\draw [ white , double=black , line width=4.44 , double distance =0.3pt] (3.5,0.17) .. controls (5.5,1.6) .. (7.5,0.17); 
\draw [ white , double=black , line width=4.44 , double distance =0.3pt] (8.5,0.17) .. controls (10.53,1.6) .. (12.5,0.17); 

\draw  [ white , double=black , line width=4.44 , double distance =0.3pt]  (2.5,0.53) .. controls (0.5,2) and (0.6,2.5) .. (1.75,2.5); 

\begin{scope}[yscale=1,xscale=-1, xshift = -6cm]
 \draw  [ white , double=black , line width=4.44 , double distance =0.3pt]  (2.5,0.53) .. controls (0.5,2) and (0.6,2.5) .. (1.75,2.5);
\end{scope} 

\begin{scope}[yscale=1,xscale=-1, xshift = -11cm]
 \draw  [ white , double=black , line width=4.44 , double distance =0.3pt]  (2.5,0.53) .. controls (0.5,2) and (0.6,2.5) .. (1.75,2.5);
\end{scope} 

\begin{scope}[xshift=5cm]
\draw  [ white , double=black , line width=4.44 , double distance =0.3pt]  (2.5,0.53) .. controls (0.5,2) and (0.6,2.5) .. (1.75,2.5);     
\end{scope} 
 
 \begin{scope}[xshift=10cm]
   \draw  [ white , double=black , line width=4.44 , double distance =0.3pt]  (2.5,0.53) .. controls (0.5,2) and (0.6,2.5) .. (1.75,2.5);
\end{scope} 

 \begin{scope}[yscale=1,xscale=-1, xshift=-3cm]    
\draw[darkred,  very thick, densely dashed] [rotate=90] (0.15, 0) parabola bend (2.2, 2.2) (3 ,1) ; 
  \end{scope}  
   
 \draw [darkred, densely dashed]  [->, very thick] (2,3) .. controls (3, 3.2) .. (4,3)node [midway, above] {$\omega_2$};
    
 \begin{scope}[xshift=3cm]    
\draw[darkred,  very thick, densely dashed] [rotate=90] (0.15, 0) parabola bend (2.2, 2.2) (3 ,1) ;  
  \end{scope}  

\begin{scope} [xshift=5cm] 
       \begin{scope}[yscale=1,xscale=-1, xshift=-3cm]    
\draw[darkred,  very thick, densely dashed] [rotate=90] (0.15, 0) parabola bend (2.2, 2.2) (3 ,1) ; 
  \end{scope}  
   
 \draw [darkred, densely dashed]  [ ->, very thick]  (2,3) .. controls (3, 3.2) .. (4,3) node [midway, above] {$\omega_4$};
    
 \begin{scope}[xshift=3cm]    
\draw[darkred,  very thick, densely dashed] [rotate=90] (0.15, 0) parabola bend (2.2, 2.2) (3 ,1) ;  
  \end{scope}  
  \end{scope}

  \begin{scope} [xshift=10cm] 
       \begin{scope}[yscale=1,xscale=-1, xshift=-3cm]    
\draw[darkred,  very thick, densely dashed] [rotate=90] (0.15, 0) parabola bend (2.2, 2.1) (3 ,1) ; 
  \end{scope}  
   
 \draw [darkred, densely dashed]  [ ->, very thick]  (2,3) .. controls (3, 3.2) .. (4,3) node [midway, above] {$\omega_6$};
    
 \begin{scope}[xshift=3cm]    
\draw[darkred,  very thick, densely dashed] [rotate=90] (0.15, 0) parabola bend (2.2, 2.2) (3 ,1) ;  
  \end{scope}  
  \end{scope}

\draw (0,0) rectangle +(1,0.7)   (2.5,0) rectangle +(1,0.7)   (5, 0) rectangle +(1,0.7)   (7.5,0) rectangle +(1,0.7)  (10, 0) rectangle +(1,0.7)  (12.5, 0) rectangle +(1,0.7)  ;
\fill[white] (0,0) rectangle +(1,0.7)   (2.5,0) rectangle +(1,0.7)   (5, 0) rectangle +(1,0.7)   (7.5,0) rectangle +(1,0.7)  (10, 0) rectangle +(1,0.7)  (12.5, 0) rectangle +(1,0.7)  ;
\draw (0.53,0.35) node {$-e_1$} (3,0.35) node {$e_2$} (5.5, 0.35) node {$-e_3$}  (8,0.35) node {$e_4$} (10.53, 0.35) node {$-e_5$}  (13, 0.35) node {$e_6$};


\end{tikzpicture} }
\caption{A Seifert surface for the knot $C(e_{1},...,e_{6})$, together with the
set of preferred generators for its first homology group.}
\label{fig:seifert_surface}
\vspace*{-6pt}
\end{figure}
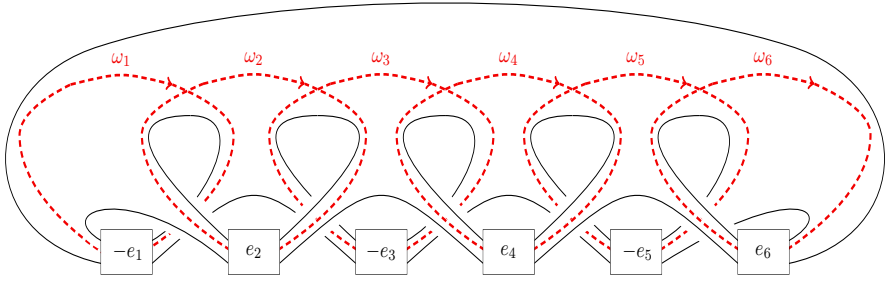

\begin{prop}\label{characteristic_knots_p=3}
A knot of the type $K_{1}(a, b)$ admits an irregular three-fold dihedral
cover if and only if

(1) $a\equiv 0 \mod 3, b\equiv 2 \mod 3$ or

(2) $a\equiv 1 \mod 3, b\equiv 1 \mod 3$.
\goodbreak

A knot of the type $K_{2}(a, b)$ admits an irregular 3-fold dihedral
cover if and only if

(3) $a\equiv 0 \mod 3, b\equiv 1 \mod 3$ or

(4) $a\equiv 1 \mod 3, b\equiv 0 \mod 3$.

In these cases, a curve representing the class $\beta \in H_{1}(V;
\mathbb{Z})$ is a mod $3$ characteristic knot for the corresponding
$K_{i}(a,b)$ if and only if there is a choice of orientation on
$\beta $ such that, with respect to the basis $\{ \omega _{1},\omega
_{2},\omega _{3}, \omega _{4},\omega _{5},\omega _{6} \}$, we have,
respectively,

(1) $[\beta ]\equiv (1, 0, 1, 1, -1, 1) \mod 3$,

(2) $[\beta ]\equiv (-1, 1, 1, 0, 1, 1)\mod 3$,

(3) $[\beta ]\equiv (1, 0, 1, -1, 1, 1)\mod 3$,

(4) $[\beta ]\equiv (-1, 1, 1, 1, 0, 1)\mod 3$.
\end{prop}

\begin{proof}
Let $V$ denote the Seifert surface for $C(e_{1},e_{2},e_{3},e_{4},e
_{5},e_{6})$ depicted in {Fig.~\ref{fig:seifert_surface}}. We think of the $e_{i}$ as being
chosen so that the knot $C(e_{1},e_{2},e_{3},e_{4},e_{5},e_{6})$ is of
type $K_{1}(a, b)$ or $K_{2}(a, b)$. Let $L$ denote the matrix of the
linking form on $V$ with respect to the basis $\{ \omega _{1},\omega
_{2},\omega _{3}, \omega _{4},\omega _{5},\omega _{6} \}$. The symmetrized
linking form for $V$ in this basis is $L_{V} = L + L^{T}$. It equals:
\vspace*{-3pt}
\[
\left (
\begin{array}{c@{\quad }c@{\quad }c@{\quad }c@{\quad }c@{\quad }c}
-e_{1} & 1 & 0 & 0 & 0 & 0
\\
1 & e_{2} & -1 & 0 & 0 & 0
\\
0 & -1 & -e_{3} & 1 & 0 & 0
\\
0 & 0 & 1 & e_{4} & -1 & 0
\\
0 & 0 & 0 & -1 & -e_{5} & 1
\\
0 &0 & 0 & 0 & 1 & e_{6}
\end{array}
\right )
\]

It is sufficient to check that $\det (L + L^{T})\equiv 0 \mod 3$
precisely in situations $(1),\allowbreak ...,(4)$. For instance, in the case
$C(e_{1},e_{2},e_{3},e_{4},e_{5},e_{6}) = K_{1}(a, b)$, we obtain
$\det (L + L^{T}) = -(8ab+2b-1)^{2}$. So we need to solve the equation
\begin{equation*}
8ab+2b-1\equiv 0\mod 3.
\end{equation*}
If $a\equiv 0$~mod~3, the equation reduces to $2b-1\equiv 0$~mod~3, so
$b\equiv 2$. If $a\equiv 1$~mod~3, then $b\equiv 1$~mod~3. If
$a\equiv 2$~mod~3, there is no solution. The computations for
$K_{1}(a, b)$ are equally trivial, so they are omitted.

To verify that the classes $[\beta ]\in H_{1}(V; \mathbb{Z})$ listed
represent all characteristic knots, it suffices to check that, for
$a$, $b$ and $\beta$ as specified, we have $(L + L^{T})\beta \equiv 0\mod 3$ and
moreover that the classes $\beta $ are the unique solutions mod~3 for each pair $(a, b)$. The
arithmetic involved has been left out.
\end{proof}

More generally, we have the following:

\begin{prop}
\label{slice_knots_have_covers}
Let $p>1$ be an odd prime. There exits an infinite family of integer
pairs $(a, b)$ such that the two-bridge slice knot $K_{1}(a, b)\subset
S^{3}$ admits an irregular dihedral $p$-fold cover, and similarly for
$K_{2}(a, b)$.
\end{prop}

\begin{proof}
The case $p=3$ was treated in
{Proposition~\ref{characteristic_knots_p=3}}, so assume $p>3$. The
determinant $D_{1}(a, b)$ of the Seifert matrix of the knot
$K_{1}(a, b)$ is equal to $-(8ab+2b-1)^{2}$. Setting $a\equiv 0
\mod p$, we find that $D_{1}(a, b)\equiv 0\mod p$ if and only if
$2b\equiv 1\mod p$. Since $p$ is odd, a solution exists. Another pair
of solutions is $a\equiv 8^{-1}\mod p$ and $b\equiv 3^{-1}\mod p$.

Similarly, we find that the determinant $D_{2}(a, b)$ of the Seifert
matrix of the knot $K_{2}(a, b)$ is $(8ab+2a+2b+1)^{2}$. Setting
$b\equiv -1\mod p$, we find that $a (-6)\equiv{1} \mod p$. For $p>3$,
this gives a solution.
\end{proof}

For any given $p$ and any family of two-bridge slice knots $K_{i}(a, b)$
with $a$ and $b$ chosen so that $\det (L + L^{T})\equiv 0 \mod p$, the
classes in $H_{1}(V; \mathbb{Z})$ represented by characteristic knots
are easily computed as in {Proposition~\ref{characteristic_knots_p=3}} by
solving a system of equations mod~$p$. One can see by direct examination
that if $p=3$ each of these homology classes can be realized by an unknot embedded in the interior of $V$. The
same methods can be used to find knot types of characteristic knots for
all $p$.%

\section*{Appendix B. Computing linking numbers in branched covers}\label{appB}

Let $\alpha \subset S^{3}$ be a knot, and let $f: M\to S^{3}$ be a cover
branched along $\alpha $, arising from a presentation $\psi : \pi _{1}(S
^{3}-\alpha , x_{0})\to S_{n}$. The linking numbers (when defined)
between the various components of $f^{-1}(\alpha )$ constitute a subtle
knot invariant studied extensively by Hartley and Murasugi
\cite{hartley1977covering}, Bankwitz and Schumann
\cite{bankwitz1934viergeflechte}, Laufer \cite{laufer1971some} and
Perko \cite{perko1974classification}, 
among others. Further applications of linking numbers in dihedral covers of knots were found by Cappell and Shaneson~\cite{CS1975invariants} and Litherland~\cite{litherland1980formula}.

In his undergraduate thesis \cite{perko1964thesis}, Perko detailed
a procedure, going back to Reidemeister~\cite{reidemeister1929knoten}, for computing linking numbers between branch curves. His
method is, to this day, the most efficient and general algorithm known
for computing these numbers. We give a very short summary of this classical 
method for computing linking numbers in a branched cover. We intend to
provide just enough detail to be able to describe a generalization of
these ideas which will allow us to calculate the linking numbers of other
curves, as needed for evaluating the component of $\Xi _{p}(\alpha )$
which is expressed in terms of linking. Readers interested in the
specifics needed to carry out the procedure can find them in~\cite{perko1964thesis} or~\cite{cahn2016linking}.

\medskip
\textit{Perko's procedure for computing linking numbers between branch
curves in a branched cover $f: M\to S^{3}$ with branching set
$\alpha $:}
\begin{enumerate}
\item
Use a diagram for $\alpha $ to endow $S^{3}$ with a cell structure. The
two-skeleton is the cone on $\alpha $, and there is a single
three-cell.
\item
Endow the cover $M$ with a cell structure as follows. The cells are the
pre-images $f^{-1}(e_{k}^{j})$ of the various cells in $S^{3}$. The
attaching maps are determined by the action of the meridians of
$\alpha $ on the interiors of the cells.
\item
Compute the boundaries of all two-cells of $M$. This step is non-trivial
for two-cells whose boundary contains one-cells corresponding to
over-arcs in the knot diagram. Such two-cells will accrue
additional boundary components determined by the action of meridians of
$\alpha $ on the three-cells.
\item
Solve a system of linear equations to determine, for each component
$\alpha _{i}$ of $f^{-1}(\alpha )$, a~two-chain with boundary
$\alpha _{i}$, if such a two-chain exists.
\item
For each pair $(\alpha _{i}, \alpha _{j})$, examine the signed
intersection numbers of $\alpha _{i}$ with a two-chain, found in (4),
whose boundary is $\alpha _{j}$. This gives $lk(\alpha _{i}, \alpha _{j})$.
We remark that, in practice, the intersection number of any one-cell with
any two-cell is trivial to read off from the data examined in order to
complete (3), so this final step of the computation poses no
difficulty.
\end{enumerate}

In order to compute the linking numbers of other curves in $M$, we
introduce an appropriate subdivision of the cell structure described above. Consider a curve $\gamma \subset (S^{3}-\alpha )$ whose lifts to
$M$ are of interest. We use the cone on $\alpha \cup \gamma $ to form
the two-skeleton of $S^{3}$. In order to lift this new cell
structure to a cell structure on $M$, we treat $\gamma $ as a
``pseudo-branch curve'' of the map $f$. That is, we think of the
homomorphism $\pi _{1}(S^{3}-\alpha )\to S_{n}$ as a homomorphism
$\pi _{1}(S^{3}-(\alpha \cup \gamma ))\to S_{n}$ in which meridians of
$\gamma $ map to the trivial permutation. Naturally, this can be done
for multiple curves $\gamma _{i}$ simultaneously. In this set up,
linking numbers can be computed by following steps (3), (4) and (5)
above. The above procedure is carried out in~\cite{cahn2016linking}, and
a computer algorithm for performing linking number calculations is
provided. \\

\textbf{Acknowledgements. }  The author is deeply indebted to her PhD
adviser, Julius Shaneson, for his support, encouragement and insight.
Thanks to humor, my rock. Mohamed-Ali Belabbas was also a sustained
source of helpful conversations. The anonymous referee made many
valuable suggestions. \\

\bibliographystyle{amsplain}
\bibliography{BrCovArt}

\providecommand{\bysame}{\leavevmode\hbox to3em{\hrulefill}\thinspace}
\providecommand{\MR}{\relax\ifhmode\unskip\space\fi MR }
\providecommand{\MRhref}[2]{%
  \href{http://www.ams.org/mathscinet-getitem?mr=#1}{#2}
}
\providecommand{\href}[2]{#2}
\begin{thebibliography}{10}

\bibitem{alexander1920note}
James Alexander, \emph{Note on {R}iemann spaces}, Bulletin of the American
  Mathematical Society \textbf{26} (1920), no.~8, 370--372.

\bibitem{bankwitz1934viergeflechte}
Carl Bankwitz and Hans~Georg Schumann, \emph{{\"U}ber viergeflechte},
  Abhandlungen aus dem Mathematischen Seminar der Universit{\"a}t Hamburg,
  vol.~10, Springer, 1934, pp.~263--284.

\bibitem{berstein1978degree}
Israel Berstein and Allan Edmonds, \emph{The degree and branch set of a
  branched covering}, Inventiones Mathematicae \textbf{45} (1978), no.~3,
  213--220.

\bibitem{cahn2016linking}
Patricia Cahn and Alexandra Kjuchukova, \emph{Linking numbers in
  three-manifolds}, arXiv preprint arXiv:1611.10330 (2016).

\bibitem{cahn2017singular}
\bysame, \emph{Singular branched covers of four-manifolds}, arXiv preprint
  arXiv:1710.11562 (2017).

\bibitem{CS1975branched}
Sylvain Cappell and Julius Shaneson, \emph{Branched cyclic coverings}, Knots,
  Groups, and 3-Manifolds (1975), 165--173.

\bibitem{CS1975invariants}
\bysame, \emph{Invariants of 3-manifolds}, Bulletin of the American
  Mathematical Society \textbf{81} (1975), no.~3, 559--562.

\bibitem{CS1984linking}
\bysame, \emph{Linking numbers in branched covers}, Contemporary Mathematics
  \textbf{35} (1984).

\bibitem{casson1986three}
Andrew Casson, \emph{Three lectures on new-infinite constructions in
  4-dimensional manifolds}, Prog. Math. (1986), 201--244.

\bibitem{cass-gor1986cobordism}
Andrew Casson and Cameron Gordon, \emph{Cobordism of classical knots}, Progress
  in Math \textbf{62} (1986), 181--199.

\bibitem{cha2016casson}
Jae~Choon Cha and Mark Powell, \emph{Casson towers and slice links},
  Inventiones mathematicae \textbf{205} (2016), no.~2, 413--457.

\bibitem{donaldson1983application}
Simon~K. Donaldson, \emph{An application of gauge theory to four-dimensional
  topology}, Journal of Differential Geometry \textbf{18} (1983), no.~2,
  279--315.

\bibitem{fox1962quick}
Ralph Fox, \emph{A quick trip through knot theory}, Topology of 3-manifolds and
  related topics \textbf{3} (1962), 120--167.

\bibitem{fox1970metacyclic}
\bysame, \emph{Metacyclic invariants of knots and links}, Canad. J. Math
  \textbf{22} (1970), 193--201.

\bibitem{fox1966singularities}
Ralph Fox and John Milnor, \emph{Singularities of 2-spheres in 4-space and
  cobordism of knots}, Osaka Journal of Mathematics \textbf{3} (1966), no.~2,
  257--267.

\bibitem{freedman1982topology}
Michael Freedman, \emph{The topology of four-dimensional manifolds}, Journal of
  Differential Geometry \textbf{17} (1982), no.~3, 357--453.

\bibitem{hartley1977covering}
Richard Hartley and Kunio Murasugi, \emph{Covering linkage invariants}, Canad.
  J. Math \textbf{29} (1977), 1312--1339.

\bibitem{hilden1974every}
Hugh Hilden, \emph{Every closed orientable 3-manifold is a 3-fold branched
  covering space of ${S}^3$}, Bulletin of the American Mathematical Society
  \textbf{80} (1974), no.~6, 1243--1244.

\bibitem{hirsch140offene}
Ulrich Hirsch, \emph{{\"U}ber offene abbildungen auf die 3-sphaic}, Math. Z
  \textbf{140} (1974), 203--230.

\bibitem{iori2002}
Massimiliano Iori and Riccardo Piergallini, \emph{4-manifolds as covers of the
  4-sphere branched over non-singular surfaces}, Geometry and Topology
  \textbf{6} (2002), no.~1, 393--401.

\bibitem{lamm2000symmetric}
Christoph Lamm, \emph{Symmetric unions and ribbon knots}, Osaka Journal of
  Mathematics \textbf{37} (2000), no.~3, 537--550.

\bibitem{lamm2006symmetric}
\bysame, \emph{Symmetric union presentations for 2-bridge ribbon knots}, arXiv
  preprint math/0602395 (2006).

\bibitem{laufer1971some}
Henry~B Laufer, \emph{Some numerical link invariants}, Topology \textbf{10}
  (1971), no.~2, 119--131.

\bibitem{lisca2007lens}
Paolo Lisca, \emph{Lens spaces, rational balls and the ribbon conjecture},
  Geometry \& Topology \textbf{11} (2007), no.~1, 429--472.

\bibitem{litherland1980formula}
Richard Litherland, \emph{A formula for the casson--gordon invariant of a
  knot}, preprint (1980).

\bibitem{montesinos1974representation}
Jos{\'e}~Mar{\'\i}a Montesinos, \emph{A representation of closed orientable
  3-manifolds as 3-fold branched coverings of ${S}^3$}, Bulletin of the
  American Mathematical Society \textbf{80} (1974), no.~5, 845--846.

\bibitem{montesinos1978}
\bysame, \emph{4-manifolds, 3-fold covering spaces and ribbons}, Transactions
  of the American mathematical society \textbf{245} (1978), 453--467.

\bibitem{morgan1984smith}
John~W Morgan and Hyman Bass, \emph{The smith conjecture},  (1984).

\bibitem{novikov1970pontrjagin}
Sergei Novikov, \emph{Pontrjagin classes, the fundamental group and some
  problems of stable algebra}, Essays on Topology and Related Topics (1970),
  147--155.

\bibitem{perko1964thesis}
Kenneth Perko, \emph{An invariant of certain knots}, Undergraduate Thesis,
  Princeton University (1964).

\bibitem{perko1974classification}
\bysame, \emph{On the classification of knots}, Proc. Am. Math. Soc \textbf{45}
  (1974), 262--266.

\bibitem{piergallini1995four}
Riccardo Piergallini, \emph{Four-manifolds as 4-fold branched covers of
  ${S}^4$}, Topology \textbf{34} (1995), no.~3, 497--508.

\bibitem{piergallini2016branched}
Riccardo Piergallini and Daniele Zuddas, \emph{On branched covering
  representation of 4-manifolds}, arXiv preprint arXiv:1602.07459 (2016).

\bibitem{reidemeister1929knoten}
Kurt Reidemeister, \emph{Knoten und verkettungen}, Mathematische Zeitschrift
  \textbf{29} (1929), no.~1, 713--729.

\bibitem{rolfsen1976knots}
Dale Rolfsen, \emph{Knots and links}, vol. 346, American Mathematical Soc.,
  1976.

\bibitem{seifert1935geschlecht}
Herbert Seifert, \emph{{\"U}ber das geschlecht von knoten}, Mathematische
  Annalen \textbf{110} (1935), no.~1, 571--592.

\bibitem{serre1961formes}
Jean-Pierre Serre, \emph{Formes bilin{\'e}aires sym{\'e}triques enti{\`e}res
  {\`a} discriminant$\pm$1}, S{\'e}minaire Henri Cartan \textbf{14} (1961),
  1--16.

\bibitem{stevens1996homology}
Wayne~H Stevens, \emph{On the homology of branched cyclic covers of knots.},
  (1996).

\bibitem{tristram1969some}
Andrew~G Tristram, \emph{Some cobordism invariants for links}, Mathematical
  Proceedings of the Cambridge Philosophical Society, vol.~66, Cambridge Univ
  Press, 1969, pp.~251--264.

\bibitem{viro1984signature}
Oleg Viro, \emph{Signature of a branched covering}, Mathematical Notes
  \textbf{36} (1984), no.~4, 772--776.

\end{thebibliography}

\end{document}